\documentclass[11pt,oneside,english]{amsart}
\usepackage{lmodern}
\usepackage[T1]{fontenc}
\usepackage[latin9]{inputenc}
\usepackage{geometry}
\usepackage{color}
\usepackage{babel}
\usepackage{verbatim}
\usepackage{amsbsy}
\usepackage{amstext}
\usepackage{mathtools}
\usepackage{amsthm}
\usepackage{amssymb}
\usepackage[numbers]{natbib}
\usepackage[unicode=true,pdfusetitle,
 bookmarks=true,bookmarksnumbered=false,bookmarksopen=false,
 breaklinks=false,pdfborder={0 0 1},backref=false,colorlinks=false]
 {hyperref}

\makeatletter
\numberwithin{equation}{section}
\numberwithin{figure}{section}
\theoremstyle{plain}
\newtheorem{lem}{\protect\lemmaname}
\theoremstyle{remark}
\newtheorem*{rem*}{\protect\remarkname}
\theoremstyle{definition}
\newtheorem{defn}{\protect\definitionname}
\theoremstyle{plain}
\newtheorem{thm}{\protect\theoremname}
\theoremstyle{definition}
 \newtheorem{example}{\protect\examplename}

\usepackage{enumitem}
\setlist{nosep}
\usepackage{graphicx}

\usepackage{pgf,tikz}
\tikzset{every picture/.style={/utils/exec={\sffamily}}}
\usetikzlibrary{arrows}
\usetikzlibrary{patterns}

\makeatother

\providecommand{\definitionname}{Definition}
\providecommand{\examplename}{Example}
\providecommand{\lemmaname}{Lemma}
\providecommand{\remarkname}{Remark}
\providecommand{\theoremname}{Theorem}

\begin{document}
\global\long\def\R{\mathbb{R}}%

\title{The Distance to the Boundary with respect to the Minkowski Functional
of a Polytope}
\author{Mohammad Safdari} 
\begin{abstract}
We study the regularity of the distance function to the boundary of
a domain in $\mathbb{R}^{n}$, with respect to the Minkowski functional of a convex polytope. We obtain the regularity of the distance function in certain cases. We also explicitly compute the distance function in a collection of examples and observe the new interesting phenomena that arise for such distance functions.  
\end{abstract}

\maketitle

\section{Introduction}

The distance function from the boundary of a domain
appears in a wide range of areas in analysis, geometry, and applied mathematics. 
\citet{MR2336304} studied
the distance function, and used it to analyze partial differential equations (PDEs) of Monge\textendash Kantorovich
type arising in optimal transport theory. \citet{MR2094267} studied
the distance function in the framework of Hamilton-Jacobi equations
and Finsler geometry. \citet{itoh2001lipschitz}, and \citet{mantegazza2003hamilton}
studied the distance function in Riemannian manifolds; and \citet{CHRUSCIEL20021}
considered the case of Lorentzian space-times. On the other hand,
\citet{clarke1995proximal} and \citet{poliquin2000local} studied
the distance function in the context of nonsmooth analysis in Hilbert
spaces. More recently, \citet{miura2024delta}
 studied the fine structure of the singular set of distance functions in Finsler manifolds; and \citet{he2023hierarchical} employed distance functions in their work on robot manipulation under uncertainty.

In {[}\citealp{Safdari20151}\nocite{safdari2017shape,MR1}\textendash \citealp{SAFDARI202176}{]},
we examined the distance function and used it to study variational
problems and PDEs with gradient constraint, along with their corresponding
free boundaries. This analysis relied heavily on the properties of the distance functions.
In particular, we were able to find an explicit
formula for the second derivatives of the distance functions and a monotonicity property for these second derivatives, which were very crucial in our analysis.
To the best of author's knowledge, formulas of this kind have not
appeared in the literature before, except for the simple case of the
Euclidean distance to the boundary. In addition, we 
completely characterized the set of singularities of the distance function using our explicit formula for its second derivative. In \cite{Paper-4}, we studied the regularity of distance functions
in two dimensions. Here we also allowed the boundary of the domain to have corners.

In this work, we examine the distance function to the boundary of a domain with respect to the Minkowski functional of a convex polytope. We obtain the regularity of the distance function in certain cases. We also explicitly compute the distance function in a collection of examples and observe the new interesting phenomena that arise for such distance functions. We hope that this initial work cultivates interest in this novel topic and motivates its further study.

\section{\protect\label{subsec:Reg_gaug}Regularity of the gauge
function}

Let $K$ be a compact convex subset of $\mathbb{R}^{n}$ whose interior
contains the origin. We recall from convex analysis (see \citep{MR3155183})
that the \textbf{gauge} function (or the \textbf{Minkowski functional})
of $K$ is the convex function 
\begin{equation}
\gamma(x)=\gamma_{K}(x):=\inf\{\lambda>0: \tfrac{1}{\lambda} x \in K\}.\label{eq:gauge}
\end{equation}
The gauge function $\gamma$ is subadditive and positively 1-homogenous:
\begin{align*}
 & \gamma(rx)=r\gamma(x),\\
 & \gamma(x+y)\le\gamma(x)+\gamma(y),
\end{align*}
for all $x,y\in\mathbb{R}^{n}$ and $r\ge0$. So $\gamma$ looks like
a norm on $\mathbb{R}^{n}$, except that $\gamma(-x)$ is not necessarily
the same as $\gamma(x)$. 
It is also easy to see that 
\[
K=\{\gamma\le1\},\qquad\partial K=\{\gamma=1\},\qquad\mathrm{int}(K)=\{\gamma<1\}.
\]

Let us also recall that the \textbf{support function} of $K$ is the
convex function 
\begin{equation}
h(x)=h_{K}(x):=\sup_{y\in K}\langle x,y\rangle,\label{eq:spt_funct}
\end{equation}
where $\langle\,,\rangle$ is the standard inner product on $\mathbb{R}^{n}$. Another notion is that of the \textbf{polar} of $K$ 
\begin{equation}
K^{\circ}:=\{x:\langle x,y\rangle\leq1\,\textrm{ for all }y\in K\}.\label{eq:K0}
\end{equation}
The set $K^{\circ}$, too, is a compact convex set containing the origin as
an interior point, and satisfies $K^{\circ\circ}=K$ (see Theorem
1.6.1 of \citep{MR3155183}). %
We will denote the gauge function and the support function of $K^{\circ}$
by 
\[
\gamma^{\circ}:=\gamma_{K^{\circ}},\qquad h^{\circ}:=h_{K^{\circ}},
\]
respectively. 
\begin{lem}
\label{lem:g_is_h0}We have 
\begin{equation}
\gamma(x)=h^{\circ}(x)=\max_{y\in K^{\circ}}\langle x,y\rangle.\label{eq:g_is_h0}
\end{equation}
\end{lem}
\begin{rem*}
Similarly we have 
\begin{equation}
\gamma^{\circ}(x)=h(x)=\max_{y\in K}\langle x,y\rangle.\label{eq:g0_is_h}
\end{equation}
\end{rem*}
\begin{proof}
First note that we can use $\max$ instead of $\sup$ since $K^{\circ}$
is compact. The equality also holds trivially for $x=0$. Now note
that for $x\ne0$ we have $x/\gamma(x)\in\partial K\subset K$. Hence,
by the definition of the polar set, for $y\in K^{\circ}$ we have
\[
\langle x/\gamma(x),y\rangle\le1\implies\langle x,y\rangle\le\gamma(x)\implies h^{\circ}(x)\le\gamma(x).
\]
On the other hand note that for every $y\in K^{\circ}$ we have 
\begin{align*}
\langle x,y\rangle\le h^{\circ}(x)\implies\langle x/h^{\circ}(x),y\rangle\le1 & \implies x/h^{\circ}(x)\in K^{\circ\circ}=K\\
 & \implies\gamma\big(x/h^{\circ}(x)\big)\le1\implies\gamma(x)\le h^{\circ}(x),
\end{align*}
which gives the desired. 
\end{proof}
As a consequence of the above lemma, for all $x,y\in\mathbb{R}^{n}$
we have 
\begin{equation}
\langle x,y\rangle\leq\gamma(x)\gamma^{\circ}(y).\label{eq:gen_Cauchy-Schwartz}
\end{equation}
In fact, more is true and we have 
\begin{equation}
\gamma^{\circ}(y)=\underset{x\ne0}{\max}\frac{\langle x,y\rangle}{\gamma(x)}.\label{eq:gen_Cauchy-Schwartz_2}
\end{equation}
To see this note that for $x\ne0$ we have $x/\gamma(x)\in\partial K$,
so 
\[
\underset{x\ne0}{\max}\frac{\langle x,y\rangle}{\gamma(x)}=\underset{x\ne0}{\max}\langle x/\gamma(x),y\rangle\le\max_{z\in K}\langle z,y\rangle=\gamma^{\circ}(y).
\]
On the other hand, for $z\in K$ we have $\gamma(z)\le1$; thus 
\[
\gamma^{\circ}(y)=\max_{z\in K}\langle z,y\rangle\le\underset{0\ne z\in K}{\max}\frac{\langle z,y\rangle}{\gamma(z)}\le\underset{x\ne0}{\max}\frac{\langle x,y\rangle}{\gamma(x)}.
\]

Next, let us review some other well-known facts from convex analysis.
Let $x\in\partial K$ and $v\in\R^{n}-\{0\}$. We say the hyperplane
\begin{equation}
H_{x,v}:=\{y\in\R^{n}:\langle y-x,v\rangle=0\}\label{eq:Hyperplane}
\end{equation}
is a \textbf{supporting hyperplane} of $K$ at $x$ if $K\subset\{y:\langle y-x,v\rangle\le0\}$.
In this case we say $v$ is an \textbf{outer normal vector} of $K$
at $x$. The \textbf{normal cone} of $K$ at $x$ is the closed convex
cone 
\begin{equation}
N(K,x):=\{0\}\cup\{v\in\R^{n}-\{0\}:v\textrm{ is an outer normal vector of }K\textrm{ at }x\}.\label{eq:Normal_cone}
\end{equation}
Since $K$ is convex, and $\emptyset\ne K\ne\R^{n}$, there is at
least one supporting hyperplane of $K$ at $x$; thus $N(K,x)$ contains
at least one nonzero element. It is also easy to see that when $\partial K$
is $C^{1}$ (which implies $\gamma$ is $C^{1}$ by Lemma \ref{lem:g_is_smooth})
we have 
\[
N(K,x)=\{tD\gamma(x):t\ge0\}.
\]
If $\dim N(K,x)\ge2$ we say $x$ is a \textbf{singular} point of
$\partial K$ (the dimension of a cone is the smallest dimension of
a subspace of $\R^{n}$ that contains the cone). Otherwise, if $\dim N(K,x)=1$
we say $x$ is a \textbf{smooth} or \textbf{regular} point of $\partial K$.
For more details see {[}\citealp{MR3155183}, Sections 1.3 and 2.2{]}.

Let us also review the definition of the \textbf{subdifferential}
of a convex function, like $\gamma$: 
\[
\partial\gamma(x):=\{v\in\R^{n}:\forall y\in\R^{n}\;\gamma(x)+\langle y-x,v\rangle\le\gamma(y)\}.
\]
The elements of the subdifferential are called \textbf{subgradients}.
It is easy to see that the subdifferential is a convex set. A convex
function such as $\gamma$ is differentiable at $x$ if and only if
$\partial\gamma(x)$ has only one element, namely $v=D\gamma(x)$;
for the proof see Theorem 1.5.15 of \citep{MR3155183}. 
\begin{lem}
For $x\ne0$ we have 
\begin{equation}
\partial\gamma(x)=N\big(K,x/\gamma(x)\big)\cap\partial K^{\circ}.\label{eq:dg}
\end{equation}
And for $x=0$ we have $\partial\gamma(0)=K^{\circ}$. 

It then follows that for $v\in\partial\gamma(x)$ we have 
\begin{equation}
\gamma^{\circ}(v)=1,\qquad\langle x,v\rangle=\gamma(x).\label{eq:v.x_is_g(x)}
\end{equation}
In particular if $\gamma$ is differentiable at $x$ we have $D\gamma(x)\in\partial K^{\circ}$,
and 
\begin{equation}
\langle D\gamma(x),x\rangle=\gamma(x).\label{eq:Euler_formula}
\end{equation}
\end{lem}
%
\begin{rem*} 
This lemma has the following consequences: 
\begin{enumerate}

\item [1.] $\gamma$ is never differentiable at $x=0$, since $\partial\gamma(0)=K^{\circ}$ has
more than one element.  

\item [2.] On the set of differentiability of $\gamma$ we have
\begin{equation}
\gamma^{\circ}(D\gamma)=1,\label{eq:g0(Dg)_is_1}
\end{equation}
and thus $D\gamma\ne0$ when it exists. 

\item [3.] We have 
\[
\partial\gamma(tx)=\partial\gamma(x)
\]
for $t>0$. In particular, if $\gamma$ is differentiable at $x\ne0$,
then it is differentiable at $tx$ for $t>0$ and we have 
\begin{equation}
D\gamma(tx)=D\gamma(x).\label{eq:Homog}
\end{equation}

\item [4.] When $z$ is a smooth point of $\partial K$ we have
\[
\partial\gamma(z)=N(K,z)\cap\partial K^{\circ}=\{tv\} \cap \partial K^{\circ}
\]
for some nonzero $v$. So $\partial\gamma(z)$ has only one element, as $\gamma^{\circ}(tv)=1$ for only one $t>0$. Therefore $\gamma$ is differentiable at
$z$. 
\end{enumerate}
\end{rem*}
\begin{proof}
For $x=0$ and $v\in\partial\gamma(0)$ we have 
\[
\langle y,v\rangle=\gamma(0)+\langle y-0,v\rangle\le\gamma(y)
\]
for every $y\in\R^{n}$. Thus by (\ref{eq:gen_Cauchy-Schwartz_2})
we must have $\gamma^{\circ}(v)\le1$ and thus $v\in K^{\circ}$.
Conversely, for $v\in K^{\circ}$ and $y\in\R^{n}$ we have 
\[
\gamma(0)+\langle y-0,v\rangle=\langle y,v\rangle\le\gamma(y)\gamma^{\circ}(v)\le\gamma(y).
\]
So we get $v\in\partial\gamma(0)$, as wanted. 

Now suppose $x\ne0$. First assume that $v\in\partial\gamma(x)$.
We know that 
\begin{equation}
\gamma(x)+\langle y-x,v\rangle\le\gamma(y)\label{eq:subdiff}
\end{equation}
for every $y\in\R^{n}$. We need to show that for $y\in K$ we have
$\langle y-x/\gamma(x),v\rangle\le0$. To see this note that by replacing
$y$ with $\gamma(x)y$ in (\ref{eq:subdiff}) we get $\langle\gamma(x)y-x,v\rangle\le\gamma\big(\gamma(x)y\big)-\gamma(x)$;
and thus 
\begin{align*}
\langle y-x/\gamma(x),v\rangle=\frac{1}{\gamma(x)}\langle\gamma(x)y-x,v\rangle & \le\frac{1}{\gamma(x)}\big[\gamma\big(\gamma(x)y\big)-\gamma(x)\big]\\
 & =\frac{1}{\gamma(x)}\big[\gamma(x)\gamma(y)-\gamma(x)\big]=\gamma(y)-1\le0,
\end{align*}
since $\gamma(y)\le1$ for $y\in K$. In addition note that if we
set $y=x+z$ in (\ref{eq:subdiff}) we obtain 
\[
\gamma(x)+\langle z,v\rangle=\gamma(x)+\langle y-x,v\rangle\le\gamma(y)=\gamma(x+z)\le\gamma(x)+\gamma(z).
\]
Hence $\langle z,v\rangle\le\gamma(z)$ for every $z$, and therefore
$\gamma^{\circ}(v)\le1$ by (\ref{eq:gen_Cauchy-Schwartz_2}). And
if we set $y=0$ we get 
\[
\gamma(x)+\langle0-x,v\rangle\le\gamma(0)\implies\gamma(x)\le\langle x,v\rangle.
\]
Thus we must have $\gamma^{\circ}(v)\ge1$, again by (\ref{eq:gen_Cauchy-Schwartz_2}).
Hence $\gamma^{\circ}(v)=1$ and therefore $v\in\partial K^{\circ}$.

Conversely, suppose $v\in N\big(K,x/\gamma(x)\big)\cap\partial K^{\circ}$.
We need to show that (\ref{eq:subdiff}) holds for every $y$. We
know that for $y\in K$ we have 
\begin{align*}
\langle y-x/\gamma(x),v\rangle\le0 & \implies\langle\gamma(x)y-x,v\rangle\le0\\
 & \implies\gamma(x)\langle y,v\rangle\le\langle x,v\rangle.
\end{align*}
However if we choose $y$ so that $\langle y,v\rangle=\max_{z\in K}\langle z,v\rangle=\gamma^{\circ}(v)=1$
(see Lemma \ref{lem:g_is_h0}), then we get $\gamma(x)\le\langle x,v\rangle$.
Thus for every $y\in\R^{n}$ we have 
\[
\gamma(x)+\langle y-x,v\rangle\le\langle x,v\rangle+\langle y-x,v\rangle=\langle y,v\rangle\le\gamma(y)\gamma^{\circ}(v)=\gamma(y),
\]
as desired. 

Finally, to prove (\ref{eq:v.x_is_g(x)}), note that if we set $y=0$
in the definition of the subgradient $v\in\partial\gamma(x)$ we obtain
$\gamma(x)-\langle x,v\rangle\le0$. On the other hand, we know that
$v\in\partial K^{\circ}$; hence $\gamma^{\circ}(v)=1$ and thus $\langle x,v\rangle\le\gamma(x)\gamma^{\circ}(v)=\gamma(x)$.
Therefore we must also have $\langle x,v\rangle=\gamma(x)$, as desired.
\end{proof}
\begin{lem}
\label{lem:x_in_dg(v)} For $x\in\partial K$
and $v\in\partial K^{\circ}$ we have 
\[
v\in\partial\gamma(x)\iff x\in\partial\gamma^{\circ}(v).
\]
As a result, for $x,v \in\R^{n} -\{0\}$ we have 
\[
\frac{v}{\gamma^{\circ}(v)}\in\partial\gamma(x)\iff\frac{x}{\gamma(x)}\in\partial\gamma^{\circ}(v).
\]
In particular, if $\gamma$ is differentiable at $x$ we have 
\[
\frac{x}{\gamma(x)}\in\partial\gamma^{\circ}\big(D\gamma(x)\big),
\]
and if in addition $\gamma^{\circ}$ is differentiable at $D\gamma(x)$
we have 
\begin{equation}
\frac{x}{\gamma(x)}=D\gamma^{\circ}\big(D\gamma(x)\big).\label{eq:Dg(Dg0)}
\end{equation}
\end{lem}
\begin{proof}
To prove the first assertion suppose $v\in\partial\gamma(x)$. We want to show that $x\in\partial\gamma^{\circ}(v)$,
i.e. for every $w$ we have 
\[
\gamma^{\circ}(v)+\langle w-v,x\rangle\le\gamma^{\circ}(w).
\]
To see this note that 
\begin{align*}
\gamma^{\circ}(v)+\langle w-v,x\rangle & =\gamma^{\circ}(v)+\langle w,x\rangle-\langle v,x\rangle\\
 & =\gamma^{\circ}(v)+\langle w,x\rangle-\gamma(x)\tag{by (\ref{eq:v.x_is_g(x)})}\\
 & =1+\langle w,x\rangle-1=\langle w,x\rangle\le\gamma^{\circ}(w)\gamma(x)=\gamma^{\circ}(w),
\end{align*}
as desired. The reverse implication can be proved similarly. The other
assertions of the lemma follow easily, noting that $\partial\gamma$
(and $D\gamma$) are positively 0-homogeneous (see (\ref{eq:Homog})). 
\end{proof}
\begin{rem*}
As a consequence of the above two lemmas, for $x\in\partial K$ and
$v\in\partial K^{\circ}$ we have 
\begin{equation}
v\in N(K,x)\iff x\in N(K^{\circ},v),\label{eq:v_in_N(x)}
\end{equation}
i.e.\@ $v$ is an outer normal vector of $K$ at $x$ if and only if
$x$ is an outer normal vector of $K^{\circ}$ at $v$. 
\end{rem*}
\begin{lem}
\label{lem:g0_diff_on_N}Let $x\in\partial K$.
If the normal cone $N(K,x)$ has nonempty interior, then an open subset
of $\partial K^{\circ}$ is flat, and $x$ is orthogonal to it. As
a result, $\gamma^{\circ}$ is differentiable on the interior of $N(K,x)$,
and for any $v\in\mathrm{int}\big(N(K,x)\big)$ we have 
\[
D\gamma^{\circ}(v)=x.
\]
\end{lem}
\begin{rem*}
This is particularly the case when $K$ is a polytope and $x$ is
a vertex of $K$ (see the following subsection on polytopes).
\end{rem*}
\begin{proof}
Let $v$ be an interior point of the cone $N(K,x)$, i.e. we have
$B_{r}(v)\subset N(K,x)$ for some $r>0$. It then easily follows
that $B_{s}(v/\gamma^{\circ}(v))\subset N(K,x)$ for $s=r/\gamma^{\circ}(v)$.
Then by (\ref{eq:v_in_N(x)}) we have $x\in N(K^{\circ},w)$ for every
$w\in B_{s}(v/\gamma^{\circ}(v))\cap\partial K^{\circ}$. This means
that $x$ is orthogonal to an open subset of $\partial K^{\circ}$,
namely $\mathrm{int}\big(N(K,x)\big)\cap\partial K^{\circ}$. Hence
this open subset of $\partial K^{\circ}$ is flat. The reason is that any two
points $w,\tilde{w}$ in this open subset must lie on the same side
of each of the parallel hyperplanes $H_{w,x},H_{\tilde{w},x}$ (the
side that $-x$ points to), which implies that $H_{w,x}=H_{\tilde{w},x}$.
Consequently, this hyperplane intersects $\partial K^{\circ}$ in
the aforementioned open subset, and thus that open subset is flat. 

Hence by the next lemma $\gamma^{\circ}$ is smooth around points
$v$ for which $v/\gamma^{\circ}(v)\in\mathrm{int}\big(N(K,x)\big)\cap\partial K^{\circ}$,
which is equivalent to $v\in\mathrm{int}\big(N(K,x)\big)$. Now by
(\ref{eq:dg}) and Lemma \ref{lem:x_in_dg(v)}
we have 
\[
v/\gamma^{\circ}(v)\in\partial\gamma(x)\implies x\in\partial\gamma^{\circ}(v)=\{D\gamma^{\circ}(v)\}\implies x=D\gamma^{\circ}(v),
\]
as desired. 
\end{proof}
\begin{lem}
\label{lem:g_is_smooth}Suppose $\partial K$ is
$C^{k,\alpha}$ $(k\ge1\,,\,0\le\alpha\le1)$ around $x_{0}$. Then
$\gamma$ is $C^{k,\alpha}$ on a neighborhood of each point $tx_{0}$
for $t>0$. 
\end{lem}
\begin{proof}
Let $r=\sigma(\theta)$, for $\theta\in\mathbb{S}^{n-1}$, be the
equation of $\partial K$ in polar coordinates. Then $\sigma$ is
positive and $C^{k,\alpha}$ around $\theta_{0}=x_{0}/|x_{0}|$. To
see this note that, locally, $\partial K$ is given by a $C^{k,\alpha}$
equation $f(x)=0$. On the other hand we have $x=rX(\theta)$, for
some smooth function $X$. Hence we have $f(rX(\theta))=0$; and the
derivative of this expression with respect to $r$ is 
\[
\bigl\langle X(\theta),Df\bigl(rX(\theta)\bigr)\bigr\rangle=\frac{1}{r}\langle x,Df(x)\rangle.
\]
But this is nonzero since $Df$ is orthogonal to $\partial K$, and
$x$ cannot be tangent to $\partial K$ (otherwise $0$ cannot be
in the interior of $K$, as $K$ lies on one side of its supporting
hyperplane at $x$). Thus we get the desired by the Implicit Function
Theorem. Now, it is straightforward to check that for a nonzero point
in $\mathbb{R}^{n}$ with polar coordinates $(s,\phi)$ we have 
\[
\gamma((s,\phi))=\frac{s}{\sigma(\phi)}.
\]
This formula easily gives the smoothness of $\gamma$ in the desired
region. 
\end{proof}

\subsection{Polytopes }

A point $z\in\partial K$ is called an \textbf{extreme point} of $K$
if it cannot be written as a convex combination of two other points
in $K$, i.e. if $z=\lambda x+(1-\lambda)y$ for some $0<\lambda<1$
and $x,y\in K$, then we must have $x=y=z$. More generally, a convex
set $F\subset\partial K$ is called a \textbf{face} of $K$ if $(x+y)/2\in F$
for $x,y\in K$ implies $x,y\in F$. The dimension of a face $F$
is the smallest dimension of an affine subspace of $\R^{n}$ that
contains $F$. The \textbf{relative interior} and \textbf{relative
boundary} of a face $F$ are the interior and the boundary of $F$
as a subset of the affine subspace of $\R^{n}$ with the smallest
dimension that contains $F$. Thus extreme points are $0$-dimensional
faces of $K$. An $(n-1)$-dimensional face is called a \textbf{facet}.%
{} For more details see {[}\citealp{MR3155183}, Section 2.1{]}.

We say $K$ is a \textbf{polytope} if $K$ has finitely many extreme
points $\{z_{1},\dots,z_{m}\}$, which are also called the \textbf{vertices}
of $K$. (Note that we are only considering $n$-dimensional polytopes
residing in $\R^{n}$, since we are only considering convex sets with
nonempty interiors.) It then follows that $K$ is the convex hull
of its finitely many vertices, i.e. every point $x\in K$ is a convex
combination of $z_{1},\dots,z_{m}$: 
\[
x=\lambda_{1}z_{1}+\dots+\lambda_{m}z_{m},\qquad\lambda_{i}\ge0,\quad\lambda_{1}+\dots+\lambda_{m}=1.
\]
Moreover, it follows that $K$ is the intersection of finitely many
closed halfspaces: 
\[
K=\bigcap_{i\le l}\{x:\langle x,v_{i}\rangle\le1\},
\]
where $v_{i}$'s are distinct. It then turns out that $K^{\circ}$
is the convex hull of $\{v_{1},\dots,v_{l}\}$; so $K^{\circ}$ is
also a polytope, and its vertices are $\{v_{1},\dots,v_{l}\}$. We
can also easily see that the facets of $K$ are 
\[
K\cap\{x:\langle x,v_{i}\rangle=1\}
\]
for $i=1,\dots,l$. Therefore, the vertices of $K^{\circ}$ are normal
vectors to the facets of $K$. Note that every point of $\partial K$
belongs to at least one facet of $K$. Also, it can be shown that
each face of $K$ is the intersection of the facets of $K$ containing
that face. In addition, each face of $K$ is the convex hull of a
subset of its vertices. For more details see {[}\citealp{MR3155183},
Section 2.4{]}. 

For a polytope $K$, the points of $\partial K$ which belong to some
face of $K$ with dimension at most $n-2$ are exactly the singular
points of $\partial K$ (see page 108 of \citep{MR3155183}). On the
other hand, the smooth points of $\partial K$ are exactly the points
lying in the relative interior of some facet of $K$. So the relative
boundaries of the facets of $K$ form the set of singular points of
$\partial K$. To see this note that by the above characterization
of $K$ and its facets, the relative boundary points of a facet $F$
must belong to the intersection of $F$ with at least one other facet,
and it is easy to see that this intersection is a face of $K$ with
dimension at most $n-2$. On the other hand, let $z$ be a point in
the relative interior of the facet $F$. Then for every $x\in F\cap B_{r}(z)$
for some small enough $r$, there is $y\in F\cap B_{r}(z)$ such that
$z=(x+y)/2$. Hence any face containing $z$ must contain $F\cap B_{r}(z)$,
and therefore must have dimension $n-1$. (Alternatively, we can see
that $K\cap B_{r}(z)$ is a half-ball, and thus $N(K,z)$ is one-dimensional.) 

Let $z$ be a vertex of a polytope $K$. We know that $z$ is a normal
vector to a facet of $K^{\circ}$. Hence by (\ref{eq:v_in_N(x)}),
for every $v$ in that facet of $K^{\circ}$ we have $v\in N(K,z)$.
In addition, no other point $w$ on $\partial K^{\circ}$ can belong
to $N(K,z)$, because then $z$ would belong to $N(K^{\circ},w)$.
However, $w$ belongs to a different facet of $K^{\circ}$, and $z$
cannot be simultaneously an outer normal vector to two facets of $K^{\circ}$.
Therefore 
\[
N(K,z)\cap\partial K^{\circ}
\]
is the facet of $K^{\circ}$ to which $z$ is normal. More generally
we have 
\begin{lem}
\label{lem:normal_to_face}Suppose $K$
is a polytope, and $z\in\partial K$. Let $F$ be the face of $K$
with smallest dimension that contains $z$, and suppose $F$ is $k$-dimensional.
Then $N(K,z)\cap\partial K^{\circ}$ is an $(n-k-1)$-dimensional
face of $K^{\circ}$. In addition, we have 
\[
N(K,z)\cap\partial K^{\circ}=\mathrm{conv}\,\{v_{i}\in\partial K^{\circ}:v_{i}\textrm{ is normal to some facet }F_{i}\supset F\},
\]
where $\mathrm{conv}\,\{v_{i}\}$ is the convex hull of $v_{i}$'s.
\end{lem}
\begin{rem*}
Note that there is a unique face of $K$ with smallest dimension that
contains $z$, because we can easily see that a nonempty intersection
of a family of faces is also a face.
\end{rem*}
\begin{rem*}
In particular, when $z$ is not a vertex and thus $k>0$, $N(K,z)\cap\partial K^{\circ}$
is a face of $K^{\circ}$ with dimension at most $n-2$, and hence
it is contained in the set of singular points of $\partial K^{\circ}$.
\end{rem*}
\begin{proof}
For the proof of the first assertion see Lemma 2.2.3 and equations
(2.25) and (2.28) in \citealp{MR3155183}. For the second assertion,
first note that by Theorem 2.4.9 of \citealp{MR3155183} we have 
\[
N(K,z)=\mathrm{pos}\,\{v_{i}:v_{i}\textrm{ is normal to some facet }F_{i}\supset F\},
\]
where $\mathrm{pos}\,\{v_{i}\}$ is the set of all positive combinations
of $v_{i}$'s, i.e.\@ the set of all points 
\[
x=\lambda_{1}v_{1}+\dots+\lambda_{j}v_{j},\qquad\lambda_{i}\ge0.
\]
Now let us further assume that $v_{i}\in\partial K^{\circ}$ (i.e.\@
assume that $v_{i}$'s are the corresponding vertices of $K^{\circ}$).
Then $N(K,z)\cap\partial K^{\circ}$ is a convex set (being the subdifferential
$\partial\gamma(z)$) that contains $v_{i}$'s, so it contains their
convex hull. And, on the other hand, for every $x\in N(K,z)\cap\partial K^{\circ}$
there are $\lambda_{i}\ge0$ so that $x=\sum\lambda_{i}v_{i}$. Now
for $\lambda=\sum\lambda_{i}>0$ (note that $\lambda=0$ implies $0=x\in\partial K^{\circ}$,
which is a contradiction) we have 
\[
y:=\frac{1}{\lambda}x=\frac{\lambda_{1}}{\lambda}v_{1}+\dots+\frac{\lambda_{j}}{\lambda}v_{j}\in N(K,z)\cap\partial K^{\circ},
\]
since $y$ is a convex combination of $v_{i}$'s. But $x\in\partial K^{\circ}$
too, thus 
\[
1=\gamma^{\circ}(y)=\frac{1}{\lambda}\gamma^{\circ}(x)=\frac{1}{\lambda}\implies\lambda=1.
\]
Hence $x$ is a convex combination of $v_{i}$'s, as desired.%
\end{proof}

\section{Regularity of the distance function}

Now let $U\subset\R^{n}$ be a bounded open set. For two points $x,y$ we denote the closed, open, and half-open
line segments with endpoints $x,y$ by $[x,y],\,]x,y[,\,[x,y[,\,]x,y]$,  respectively. We define the distance
to $\partial U$ with respect to $\gamma$ as follows 
\begin{equation}
\rho(x)=d_{K}(x,\partial U):=\underset{y\in\partial U}{\min}\,\gamma(x-y).\label{eq:rho}
\end{equation}
It is well known (see {[}\citealp{MR667669}, Section 5.3{]}) that
$\rho$ is the unique viscosity solution of the Hamilton-Jacobi equation
\begin{equation}
\begin{cases}
\gamma^{\circ}(D\rho)=1 & \textrm{in }U,\\
\rho=0 & \textrm{on }\partial U.
\end{cases}\label{eq:H-J_eq}
\end{equation}
Moreover, from the definition of $\rho$ we easily obtain 
\begin{equation}
-\gamma(x-\tilde{x})\le\rho(\tilde{x})-\rho(x)\le\gamma(\tilde{x}-x),\label{eq:rho_Lip}
\end{equation}
for all $x,\tilde{x}\in\R^{n}$. Thus, in particular, $\rho$ is Lipschitz
continuous. 
\begin{defn}
When $\rho(x)=\gamma(x-y)$ for some $y\in\partial U$, we call $y$
a \textbf{$\boldsymbol{\rho}$-closest} point to $x$ on $\partial U$. 
\end{defn}
\begin{rem*}
Note that $y\in\partial U$ is the unique $\rho$-closest point on
$\partial U$ to itself.
\end{rem*}
\begin{lem}
\label{lem:K_ball_touch_bdry}Let $x\in U$,
and consider the convex set 
\[
K_{x}:=x-\rho(x)K:=\{x-\rho(x)z:z\in K\}.
\]
Then $\mathrm{int}(K_{x})\subset U$, and $\partial K_{x}\cap\partial U$
equals the set of $\rho$-closest points on $\partial U$ to $x$. 

Conversely, if the convex set $L:=x-rK$ satisfies $\mathrm{int}(L)\subset U$
and $\partial L\cap\partial U\ne\emptyset$, then we must have $L=K_{x}$.%
\end{lem}
\begin{rem*}
It also follows that $K_{x}\subset\overline{U}$, since $K$, and
thus $K_{x}$ is the closure of its interior.
\end{rem*}
\begin{rem*}
As a consequence, when $y$ is a $\rho$-closest point on $\partial U$
to $x$ we have $[x,y[\,\subset U$. 
\end{rem*}
\begin{proof}
It is easy to see that 
\[
\partial K_{x}=\{x-\rho(x)z:z\in\partial K\},\qquad\mathrm{int}(K_{x})=\{x-\rho(x)z:z\in\mathrm{int}(K)\}.
\]
Now for any point $y\in\partial U$ we have
\[
\gamma(x-y)=\rho(x)\iff\gamma\Big(\frac{y-x}{-\rho(x)}\Big)=\gamma\Big(\frac{x-y}{\rho(x)}\Big)=1\iff z=\frac{y-x}{-\rho(x)}\in\partial K\iff y\in\partial K_{x}.
\]
Thus $\partial K_{x}\cap\partial U$ equals the set of $\rho$-closest
points on $\partial U$ to $x$. In addition, $\mathrm{int}(K_{x})$
does not intersect $\R^{n}-U$, since otherwise there would have been
some $\tilde{y}\in\partial U\cap\mathrm{int}(K_{x})$ for which we have 
\[
\tilde{z}=\frac{\tilde{y}-x}{-\rho(x)}\in\mathrm{int}(K)\implies\gamma\Big(\frac{\tilde{y}-x}{-\rho(x)}\Big)<1\implies\gamma(x-\tilde{y})<\rho(x),
\]
which is a contradiction. 

Next consider the set $L$. If $r>\rho(x)$ then we get $K_{x}\subset\mathrm{int}(L)$,
and thus $\mathrm{int}(L)$ intersects $\partial U$. And if $r<\rho(x)$
then we get $L\subset\mathrm{int}(K_{x})$; so $\partial L$ cannot
intersect $\partial U$. Therefore we must have $r=\rho(x)$, as desired. 
\end{proof}

\begin{lem}
\label{lem:cont_of_y}Suppose $x_{i}\in\overline{U}$
converges to $x\in\overline{U}$, and $y_{i}\in\partial U$ is a (not
necessarily unique) $\rho$-closest point to $x_{i}$.
\begin{enumerate}
\item[\upshape{(a)}] If $y_{i}$ converges to $\tilde{y}\in\partial U$, then $\tilde{y}$
is one of the $\rho$-closest points on $\partial U$ to $x$.
\item[\upshape{(b)}] If $y\in\partial U$ is the unique $\rho$-closest point to $x$,
then $y_{i}$ converges to $y$.
\end{enumerate}
\end{lem}
\begin{proof}
This lemma is a simple consequence of the continuity of $\gamma,\rho$,
and compactness of $\partial U$. For (a) we have 
\[
\gamma(x-\tilde{y})=\lim\gamma(x_{i}-y_{i})=\lim\rho(x_{i})=\rho(x).
\]
Hence $\tilde{y}$ is a $\rho$-closest point to $x$.

Now to prove (b) suppose to the contrary that $y_{i}\not\to y$. Then
as $\partial U$ is compact, there is a subsequence $y_{i_{k}}$ that
converges to $z\in\partial U$ where $z\ne y$. Then by (a) $z$ must
be a $\rho$-closest point to $x$, which is in contradiction with
our assumption.
\end{proof}
\begin{lem}
\label{lem:segment_to_the_closest_pt}Suppose
$y$ is one of the $\rho$-closest points on $\partial U$ to $x\in U$.
Then $y$ is a $\rho$-closest point on $\partial U$ to every point
of $[x,y]$. Therefore $\rho$ varies linearly along the line segment
$[x,y]$.%
{} 

Furthermore, if $y$ is the unique $\rho$-closest point on $\partial U$
to $x$, then it is the unique $\rho$-closest point on $\partial U$
to every point of $[x,y]$. 
\end{lem}
\begin{proof}
Let $z\in[x,y]$. Then we have $z-y=t(x-y)$ for some $t\in[0,1]$.
Suppose to the contrary that there is $\tilde{y}\in\partial U-\{y\}$
such that 
\[
\gamma(z-\tilde{y})<\gamma(z-y).
\]
Then we have 
\begin{align*}
\gamma(x-\tilde{y}) & \le\gamma(x-z)+\gamma(z-\tilde{y})<\gamma(x-z)+\gamma(z-y)\\
 & =\gamma\big((1-t)(x-y)\big)+\gamma\big(t(x-y)\big)=(1-t+t)\gamma(x-y)=\gamma(x-y),
\end{align*}
which is a contradiction. Hence $y$ is a $\rho$-closest point to
$z$.

Therefore the points in the segment $[x,y]$ have $y$ as a $\rho$-closest
point on $\partial U$. Hence for $0\le t\le\gamma(x-y)$ we have
\begin{align*}
\rho\Big(x-\frac{t}{\gamma(x-y)}(x-y)\Big) & =\gamma\Big(x-\frac{t}{\gamma(x-y)}(x-y)-y\Big)\\
 & =\Big(1-\frac{t}{\gamma(x-y)}\Big)\gamma(x-y)=\gamma(x-y)-t.
\end{align*}
Thus $\rho$ varies linearly along the segment.

For the last assertion, we can repeat the above argument starting
from $\gamma(z-\tilde{y})=\gamma(z-y)$ and arriving at $\gamma(x-\tilde{y})\le\gamma(x-y)$,
which is again a contradiction. 
\end{proof}

\begin{lem} 
Let $x\in U$, and suppose that $y$ is one of the $\rho$-closest points
to $x$ on $\partial U$. Suppose $\partial U$ is $C^{1}$ and $\nu$
is the inward normal to $\partial U$. Then we have 
\begin{equation}
\frac{x-y}{\gamma(x-y)}\in\partial\gamma^{\circ}(\nu(y)).\label{eq:K-normal}
\end{equation}
In particular, if $\gamma^{\circ}$ is differentiable at $\nu(y)$
we have 
\begin{equation}
x=y+\rho(x)\,D\gamma^{\circ}(\nu(y)).\label{eq:Parametrize_by_rho}
\end{equation}
\end{lem}
\begin{rem*}
Consequently, (by (\ref{eq:dg}) and (\ref{eq:v_in_N(x)}), or directly
from the following proof) we also have 
\begin{equation*}
\nu(y)\in N\Big(K,\frac{x-y}{\rho(x)}\Big).
\end{equation*}
In addition, we will see that 
\begin{equation}
-\nu(y)\in N(K_x,y).\label{eq:-nu_in_Kx}
\end{equation}
\end{rem*}
\begin{rem*}
If instead of $\partial U$ being $C^{1}$, we merely assume that
locally $U$ is on one side (the side to which $\nu$ is pointing)
of a hypersurface tangent at $y$ to the hyperplane $H_{y,\nu}$,
then the following proof works, and the conclusion of the lemma still
holds.
\begin{figure}
\begin{tikzpicture}[line cap=round,line join=round,>=triangle 45,x=1.0cm,y=1.0cm] 

\clip(-5,-1) rectangle (8.,4.5); 

\draw [line width=0.75pt] (0.74,0.509615431964962) -- (3.14,1.3) -- (3.34,4) -- (1.34,3.5) -- (0.34,1.4) -- cycle;

\draw [rotate around={-21:(0.74,0.509615431964962)}, line width=0.3pt,dash pattern=on 2pt off 2pt] (0.74,0.509615431964962)-- (1.6350559584448812,3.3863653478797637); 

\draw [rotate around={-21:(0.74,0.509615431964962)}, ->,line width=0.3pt, >=stealth] (0.74,0.509615431964962) -- (1.1591475422042112,1.8567743949496367); 

\draw [rotate around={-28:(0.74,0.509615431964962)}, ->,line width=0.3pt, >=stealth] (0.74,0.509615431964962) -- (0.40621183416745316,1.4974466417936942); 

\draw [rotate around={-18:(0.74,0.509615431964962)}, line width=0.5pt]  (-3.5479974267239784,-0.9673560492307696).. controls (-2,1) and (4,0) ..  (6.8174727583718076,2.6029594918089893);

\begin{scriptsize} 

\draw [fill=black] (0.74,0.509615431964962) circle (0.5pt); 

\draw[color=black] (0.82,0.13) node {$y$};

\draw[rotate around={-21:(0.74,0.509615431964962)}, color=black] (1.84,3.63) node {$x$}; 

\draw [rotate around={-21:(0.74,0.509615431964962)}, fill=black] (1.6350559584448812,3.3863653478797637) circle (0.5pt);

\draw[color=black] (2.2,1.4) node {$D\gamma ^\circ (\nu)$}; 

\draw[rotate around={-30:(0.74,0.509615431964962)}, color=black] (0.34,1.03) node {$\nu$};

\draw[rotate around={-30:(0.74,0.509615431964962)}, color=black] (-1.44,-1) node {$\partial U$}; 

\draw[color=black] (4.5,3) node {$x - \rho (x) K$}; 

\end{scriptsize} 
\end{tikzpicture} 

\caption{\protect\label{fig:1}$y$ is the $\rho$-closest point to
$x$ on $\partial U$. Note that $K$ went through a point reflection, then scaled and translated to give $K_x =x-\rho(x)K$.}
\end{figure}
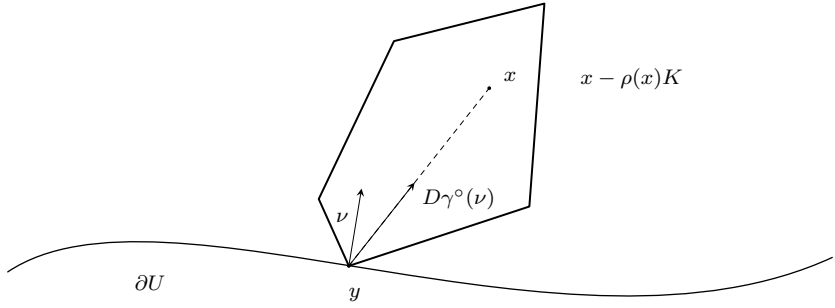
\end{rem*}
\begin{proof}
Consider the convex set $K_{x}=x-\rho(x)K$. We know that $\mathrm{int}(K_{x})\subset U$
and $y\in\partial K_{x}\cap\partial U$. The hyperplane $H_{y,\nu}$
is tangent to $\partial U$ at $y$, so we must have 
\[
K_{x}\subset\{z:\langle z-y,\nu\rangle\ge0\}.
\]
Because otherwise we would have $\langle z-y,\nu\rangle<0$ for some
$z\in K_{x}$. Then we get 
\[
\langle z_{t}-y,\nu\rangle<0
\]
for $z_{t}:=tz+(1-t)y\in\,]y,z]$. However this implies that the line
segment $]y,z]$ intersects the exterior of $U$, since $H_{y,\nu}$
is tangent to $\partial U$ at $y$ and $\nu$ points inward. But
$]y,z]\subset K_{x}$ as $K_{x}$ is convex; so in this case $K_{x}$
will intersect the exterior of $U$, which is a contradiction. 

Therefore $K_{x}\subset\{z:\langle z-y,\nu\rangle\ge0\}$, and thus $-\nu\in N(K_x,y)$. Hence for every $w\in K$ we have 
\begin{align*}
x-\rho(x)w\in\{z:\langle z-y,\nu\rangle\ge0\} & \implies\langle x-\rho(x)w-y,\nu\rangle\ge0\\
 & \implies\langle w-\frac{x-y}{\rho(x)},\nu\rangle\le0\implies\nu\in N\Big(K,\frac{x-y}{\rho(x)}\Big).
\end{align*}
(Note that $\frac{x-y}{\rho(x)}=\frac{x-y}{\gamma(x-y)}\in\partial K$.)
Hence by (\ref{eq:dg}) we have 
\[
\frac{\nu}{\gamma^{\circ}(\nu)}\in\partial\gamma\Big(\frac{x-y}{\gamma(x-y)}\Big),
\]
and thus by Lemma \ref{lem:x_in_dg(v)} we get
$\frac{x-y}{\gamma(x-y)}\in\partial\gamma^{\circ}(\nu)$, as desired. 
\end{proof}
For two vectors $a,b$, the $a\otimes b$ denotes the rank-one matrix
whose action on a vector $z$ is $\langle z,b\rangle a$. It is easy
to see that the transpose of this matrix is $(a\otimes b)^{T}=b\otimes a$.
Now let us define 
\begin{equation}
X=X(v):=\frac{1}{\langle D\gamma^{\circ}(v),v\rangle}D\gamma^{\circ}(v)\otimes v=\frac{1}{\gamma^{\circ}(v)}D\gamma^{\circ}(v)\otimes v,\label{eq:X}
\end{equation}
provided that $\gamma^{\circ}$ is differentiable at $v$. Using (\ref{eq:Euler_formula})
we can easily show that $X^{2}=X$, and 
\[
(I-X)D\gamma^{\circ}(v)=0,
\]
where $I$ is the identity matrix. 

When $K$ is a polytope, for almost every point $y$ on the boundary of a generic smooth domain, we expect the inward unit normal $\nu(y)$ to belong to the interior of a normal cone at some vertex of $K$. The next theorems show that $\rho$ is smooth around these points $y$ and around any point $x$ that has such a $y$ as its unique $\rho$-closest point on the boundary. It should be noted that the set of $x$'s for which this condition fails can have a positive measure, as we will see in Example \ref{exa:d_infty}. However, the following theorems do not require $K$ to be a polytope and are stated in more general terms. 
%

\begin{thm}
\label{thm:rho_is_C^2}Suppose $\partial U$ is
$C^{k,\alpha}$, where $k\ge1$ and $0\le\alpha\le1$, and let $\nu$
be the inward unit normal to $\partial U$. Let $x_{0}\in U$ and
suppose $y_{0}\in\partial U$ is the unique $\rho$-closest point
to $x_{0}$. Furthermore, suppose $\nu_{0}=\nu(y_{0})$ is an interior
point of the normal cone $N(K,\frac{x_{0}-y_{0}}{\rho(x_{0})})$.
Then $\rho$ is $C^{k,\alpha}$ on a neighborhood of $x_{0}$, and
every point $x$ near $x_{0}$ has a unique $\rho$-closest point
$y\in\partial U$ that is near $y_{0}$. In addition, we have 
\begin{equation}
D\rho(x)=\mu(y):=\frac{\nu(y)}{\gamma^{\circ}(\nu(y))}.\label{eq:D_rho(x)}
\end{equation}
Also, $y$ is a $C^{k,\alpha}$ function of $x$, $\gamma^{\circ}$
is differentiable at $\nu=\nu(y)$, and we have 
\begin{equation}
Dy(x)=I-X(\nu)=I-\frac{1}{\gamma^{\circ}(\nu)}D\gamma^{\circ}(\nu)\otimes\nu.\label{eq:Dy(x)}
\end{equation}
Furthermore, when $k\ge2$ we have 
\begin{equation}
D^{2}\rho(x)=\frac{1}{\gamma^{\circ}(\nu)}(I-X^{T})D^{2}d(y)(I-X),\label{eq:D2_rho(x)}
\end{equation}
where $d(\cdot):=\min_{z\in\partial U}|\cdot-\,z|$ is the Euclidean
distance to $\partial U$.
\end{thm}
\begin{rem*}
Note that, in particular, $D\rho(x)\ne0$. 
\end{rem*}
\begin{rem*}
It is well known that we have $\nu=Dd$. So we can regard $Dd$ as
a $C^{k-1,\alpha}$ extension of $\nu$ to a neighborhood of $\partial U$.
Let us also recall that the eigenvalues of $D^{2}d(y)=D\nu(y)$ are
minus the principal curvatures of $\partial U$ at $y$, together
with one eigenvalue 0. For the details see {[}\citealp{MR1814364},
Section 14.6{]}.
\end{rem*}
\begin{thm}
\label{thm:rho_is_C2_at_y}Suppose
$\partial U$ is $C^{k,\alpha}$, where $k\ge1$ and $0\le\alpha\le1$,
and let $\nu$ be the inward unit normal to $\partial U$. Let $y_{0}\in\partial U$
and suppose there is $x_{0}\in U$ such that $y_{0}$ is the unique%
{} $\rho$-closest point on $\partial U$ to $x_{0}$. Furthermore,
suppose $\nu_{0}=\nu(y_{0})$ is an interior point of the normal cone
$N(K,\frac{x_{0}-y_{0}}{\rho(x_{0})})$. Then there is an open ball
$B_{r}(y_{0})$ such that $\rho$ is $C^{k,\alpha}$ on $\overline{U}\cap B_{r}(y_{0})$.
In addition, every $y\in \partial U \cap B_{r}(y_{0})$ is the unique
$\rho$-closest point on $\partial U$ to some points in $U$. Moreover,
we have 
\begin{equation}
D\rho(y)=\mu(y)=\frac{\nu(y)}{\gamma^{\circ}(\nu(y))}.\label{eq:D_rho(y)}
\end{equation}
Furthermore, when $k\ge2$ we have 
\begin{equation}
D^{2}\rho(y)=\frac{1}{\gamma^{\circ}(\nu)}(I-X^{T})D^{2}d(y)(I-X),\label{eq:D2_rho(y)}
\end{equation}
where $d$ is the Euclidean distance to $\partial U$, and $X=X(\nu)$
is given by (\ref{eq:X}).
\end{thm}
\begin{rem*}
Consequently, when $x$ has $y$ as its unique $\rho$-closest point
on $\partial U$ we have 
\[
D^{2}\rho(x)=D^{2}\rho(y).
\]
Hence $D^{2}\rho$ is constant along the segment $[x,y]$, since by
Lemma \ref{lem:segment_to_the_closest_pt}
the points on the segment have $y$ as their unique $\rho$-closest
point. This is a special case of the monotonicity formula for the
second derivative of distance functions explored in \citep{SAFDARI202176}.
\end{rem*}
\begin{proof}[\textbf{Proof of Theorem \ref{thm:rho_is_C^2}}]
 Let $\mathrm{z}\mapsto Y(\mathrm{z})$ be a $C^{k,\alpha}$ parametrization
of $\partial U$ around $Y(0)=y_{0}$, where $\mathrm{z}$ varies
in an open set $V\subset\R^{n-1}$. Consider the map $G:V\times\R\to\R^{n}$
defined by 
\[
G(\mathrm{z},t):=Y(\mathrm{z})+\frac{t}{\rho(x_{0})}(x_{0}-y_{0}).
\]
Note that $G$ is a $C^{k,\alpha}$ function. Also note that we have
$G(0,\rho_{0})=x_{0}$, where $\rho_{0}=\rho(x_{0})$. Now we have
\[
\begin{cases}
D_{\mathrm{z}_{j}}G=D_{\mathrm{z}_{j}}Y,\\
D_{t}G=\frac{1}{\rho(x_{0})}(x_{0}-y_{0})=D\gamma^{\circ}(\nu_{0}).
\end{cases}
\]
Note that $\gamma^{\circ}$ is differentiable at $\nu_{0}$ by Lemma
\ref{lem:g0_diff_on_N}, and thus by (\ref{eq:Parametrize_by_rho})
we have $\frac{x_{0}-y_{0}}{\rho_{0}}=D\gamma^{\circ}(\nu_{0})$.
Let $w_{j}:=D_{\mathrm{z}_{j}}Y$. Note that at $\mathrm{z}=0$ the
vectors $w_{1},\dots,w_{n-1}$ form a basis for the tangent space
to $\partial U$ at $y_{0}$. Let $w$ be the orthogonal projection
of $D\gamma^{\circ}(\nu_{0})$ on this tangent space. Then we have
(we represent a matrix by its columns) 
\begin{align*}
 & \hspace{-1cm}\det DG(0,\rho_{0})\\
 & =\det\begin{bmatrix}w_{1} & \cdots & w_{n-1} & D\gamma^{\circ}(\nu_{0})\end{bmatrix}=\det\begin{bmatrix}w_{1} & \cdots & w_{n-1} & w+\langle D\gamma^{\circ}(\nu_{0}),\nu_{0}\rangle\nu_{0}\end{bmatrix}\\
 & =\langle D\gamma^{\circ}(\nu_{0}),\nu_{0}\rangle\det\begin{bmatrix}w_{1} & \cdots & w_{n-1} & \nu_{0}\end{bmatrix}=\gamma^{\circ}(\nu_{0})\det\begin{bmatrix}w_{1} & \cdots & w_{n-1} & \nu_{0}\end{bmatrix}\ne0.
\end{align*}
Note that in the last line we have used (\ref{eq:Euler_formula}),
and the fact that $w_{1},\dots,w_{n-1},\nu_{0}$ are linearly independent. 

Therefore, by the inverse function theorem, $G$ is invertible on an
open set of the form $W\times(\rho_{0}-h,\rho_{0}+h)$, and it has
a $C^{k,\alpha}$ inverse on $B_{r}(x_{0})\subset U$. We can assume
that $W$ is small enough so that for $y\in Y(W)$ we have $\nu(y)\in\mathrm{int}\big(N(K,\frac{x_{0}-y_{0}}{\rho_{0}})\big)$,
since $\nu$ is continuous. Suppose $r$ is small enough so that for
every $x\in B_{r}(x_{0})$ the $\rho$-closest points on $\partial U$
to $x$ belong to $Y(W)$. This is possible due to Lemma \ref{lem:cont_of_y}
and the fact that $y_{0}$ is the unique $\rho$-closest point to
$x_{0}$ by assumption. Also suppose that $r$ is small enough so
that for every $x\in B_{r}(x_{0})$ we have $\rho(x)\in(\rho_{0}-h,\rho_{0}+h)$,
which is possible due to the continuity of $\rho$. Now we know that
$G:(\mathrm{z},t)\mapsto x$ has an inverse, denoted by $\mathrm{z}(x),t(x)$,
where $\mathrm{z}(\cdot),t(\cdot)$ are $C^{k,\alpha}$ functions
of $x$. Let $y:=Y(\mathrm{z}(x))$. Then we have 
\[
x=G(\mathrm{z}(x),t(x))=y+t(x)\,D\gamma^{\circ}(\nu_{0}).
\]
On the other hand, (\ref{eq:Parametrize_by_rho}) implies that 
\[
x=\hat{y}+\rho(x)\,D\gamma^{\circ}(\nu(\hat{y})),
\]
where $\hat{y}$ is one of the $\rho$-closest points on $\partial U$
to $x$, which by our assumption about $B_{r}(x_{0})$ must belong
to $Y(W)$. Also note that by our assumption about $W$ we have $\nu_{0},\nu(\hat{y})\in\mathrm{int}\big(N(K,\frac{x_{0}-y_{0}}{\rho_{0}})\big)$;
so Lemma \ref{lem:g0_diff_on_N} implies
that $\gamma^{\circ}$ is differentiable at $\nu(\hat{y})$ and we
have 
\[
D\gamma^{\circ}(\nu(\hat{y}))=D\gamma^{\circ}(\nu_{0}).
\]
Thus we get 
\[
x=\hat{y}+\rho(x)\,D\gamma^{\circ}(\nu_{0}).
\]
But by our assumption about $B_{r}(x_{0})$, there is $\hat{\mathrm{z}}\in W$
such that $\hat{y}=Y(\hat{\mathrm{z}})$. Hence $(\hat{\mathrm{z}},\rho(x))\in W\times(\rho_{0}-h,\rho_{0}+h)$,
and we have $G(\hat{\mathrm{z}},\rho(x))=x$. Therefore due to the
invertibility of $G$ we must have 
\[
\hat{\mathrm{z}}=\mathrm{z}(x),\qquad\rho(x)=t(x).
\]
Hence $\rho$ is a $C^{k,\alpha}$ function on $B_{r}(x_{0})$. In
addition, it follows that $x$ has a unique $\rho$-closest point
on $\partial U$ given by $y=Y(\mathrm{z}(x))$, because $\hat{y}=Y(\hat{\mathrm{z}})=Y(\mathrm{z}(x))=y$.
Also, $y$ is a $C^{k,\alpha}$ function of $x$, since $Y,\mathrm{z}$
are $C^{k,\alpha}$ functions.

Next remember that 
\begin{align*}
DG(\mathrm{z},t)=A & :=\begin{bmatrix}w_{1}(\mathrm{z}) & \cdots & w_{n-1}(\mathrm{z}) & D\gamma^{\circ}(\nu_{0})\end{bmatrix}\\
 & =\begin{bmatrix}w_{1}(\mathrm{z}) & \cdots & w_{n-1}(\mathrm{z}) & D\gamma^{\circ}(\nu)\end{bmatrix}\negthickspace,
\end{align*}
where $\nu=\nu(Y(\mathrm{z}))$ (note that $D\gamma^{\circ}(\nu)=D\gamma^{\circ}(\nu_{0})$
by Lemma \ref{lem:g0_diff_on_N}). Similarly
to the calculation of $\det DG(0,\rho_{0})$, we can show that $A$
is invertible. Therefore we have 
\[
DG^{-1}(x)=A^{-1}.
\]
It is easy to see that the $n$th row of $A^{-1}$ is $\frac{1}{\gamma^{\circ}(\nu)}\nu$,
since the product of this row with every column $w_{j}$ is zero as
$\nu$ is orthogonal to $w_{j}$'s, and the product of this row with
the column $D\gamma^{\circ}(\nu)$ is 1 by (\ref{eq:Euler_formula}).
In addition, note that $Dt(x)$ is the $n$th row of $DG^{-1}$; therefore
(noting that $\rho(x)=t(x)$) 
\[
D\rho(x)=Dt(x)=\frac{\nu}{\gamma^{\circ}(\nu)}=\mu(y),
\]
where $y=Y(\mathrm{z}(x))$ is the unique $\rho$-closest point on
$\partial U$ to $x$. Also note that when $i<n$, the $i$th component
of $G^{-1}$ is $\mathrm{z}_{i}$. Hence the $i$th row of $DG^{-1}$
is $D\mathrm{z}_{i}$. On the other hand, the $i$th row of $DG^{-1}$
is equal to $e_{i}^{T}A^{-1}$, where $e_{i}$ is the $i$th standard
basis (column) vector in $\R^{n}$. 
So we have $D\mathrm{z}=\tilde{I}A^{-1}$, where $\tilde{I}$ is the
$(n-1)\times n$ matrix whose $i$th row is $e_{i}^{T}$. 

Now we have $Dy(x)=DY(\mathrm{z})D\mathrm{z}(x)$. On the other hand
we know that $DY=\begin{bmatrix}w_{1} & \cdots & w_{n-1}\end{bmatrix}$,
i.e. the $j$th column of $DY$ is the $j$th column of $A$, for
$j<n$. Then it is easy to check that $DY(\mathrm{z})\tilde{I}=A\hat{I}$,
where $\hat{I}$ is the $n\times n$ matrix whose first $n-1$ columns
are the same as $I$, and its $n$th column is $0$. Next note
that the $n$th row of $A^{-1}$ is $\frac{1}{\gamma^{\circ}(\nu)}\nu$.
Hence we have 
\begin{align*}
Dy=DYD\mathrm{z} & =DY\tilde{I}A^{-1}=A\hat{I}A^{-1}=A\big(I-\begin{bmatrix}0 & \cdots & 0 & e_{n}\end{bmatrix}\big)A^{-1}\\
 & =I-A\begin{bmatrix}0 & \cdots & 0 & e_{n}\end{bmatrix}A^{-1}=I-\begin{bmatrix}0 & \cdots & 0 & Ae_{n}\end{bmatrix}A^{-1}\\
 & =I-\begin{bmatrix}0 & \cdots & 0 & D\gamma^{\circ}(\nu)\end{bmatrix}A^{-1}=I-\frac{1}{\gamma^{\circ}(\nu)}D\gamma^{\circ}(\nu)\otimes\nu=I-X,
\end{align*}
as wanted. 

Finally note that 
\[
D\rho(x)=\mu(y)=\frac{\nu(y)}{\gamma^{\circ}(\nu(y))}=\frac{Dd(y)}{\gamma^{\circ}(Dd(y))}=\frac{Dd(y(x))}{\gamma^{\circ}\big(Dd(y(x))\big)}.
\]
Thus by differentiating this equality we obtain 
\begin{align*}
D^{2}\rho(x) & =\Big[\frac{1}{\gamma^{\circ}(\nu)}I-\frac{1}{(\gamma^{\circ}(\nu))^{2}}Dd(y)\otimes D\gamma^{\circ}(\nu)\Big]D^{2}d(y)Dy(x)\\
 & =\frac{1}{\gamma^{\circ}(\nu)}\Big[I-\frac{1}{\gamma^{\circ}(\nu)}\nu(y)\otimes D\gamma^{\circ}(\nu)\Big]D^{2}d(y)(I-X)\\
 & =\frac{1}{\gamma^{\circ}(\nu)}(I-X^{T})D^{2}d(y)(I-X),
\end{align*}
which is the desired formula for $D^{2}\rho(x)$.
\end{proof}
\begin{proof}[\textbf{Proof of Theorem \ref{thm:rho_is_C2_at_y}}]
 We will show that $\rho$ has a $C^{k,\alpha}$ extension to an
open neighborhood of $y$. Note that if we consider $\rho$ as a function
on all of $\R^{n}$, then it is not differentiable on $\partial U$.
However, we will show that the following extension of $\rho$, which
can be considered a signed version of $\rho$ on $\R^{n}$, is $C^{k,\alpha}$
on a neighborhood of $y$: 
\[
\rho_{s}(x):=\begin{cases}
\rho(x) & \textrm{if }x\in\overline{U},\\
-\bar{\rho}(x) & \textrm{if }x\in\R^{n}-\overline{U}.
\end{cases}
\]
Here $\bar{\rho}:=d_{-K}(\cdot,\partial U)$ (note that $\gamma_{-K}(\cdot)=\gamma(-\,\cdot)$).
Notice that when $x\in\partial U$ we have $-\bar{\rho}(x)=0=\rho(x)$.
So in particular, $\rho_{s}$ is a continuous function. In addition,
note that $\partial(\R^{n}-\overline{U})=\partial U$, but the inward
unit normal to $\partial(\R^{n}-\overline{U})$ is $-\nu$. Now note
that the gauge function of $(-K)^{\circ}=-K^{\circ}$ is $\gamma^{\circ}(-\,\cdot)$.
Thus if we incorporate this in (\ref{eq:Parametrize_by_rho}), we
obtain 
\begin{equation}
x=y+\bar{\rho}(x)\,\big(-D\gamma^{\circ}(-(-\nu))\big)=y-\bar{\rho}(x)\,D\gamma^{\circ}(\nu),\label{eq:Parametrize_by_-rho}
\end{equation}
provided that $x\in\R^{n}-\overline{U}$ has $y$ as its $\bar{\rho}$-closest
point on $\partial U$ and $\gamma^{\circ}$ is differentiable at
$\nu=\nu(y)$. Also note that the derivative of the gauge function
of $(-K)^{\circ}$ is $-D\gamma^{\circ}(-\,\cdot)$. 

Let $\mathrm{z}\mapsto Y(\mathrm{z})$ be a $C^{k,\alpha}$ parametrization
of $\partial U$ around $Y(0)=y_{0}$, where $\mathrm{z}$ varies
in an open set $V\subset\R^{n-1}$. Consider the map $G:V\times\R\to\R^{n}$
defined by 
\[
G(\mathrm{z},t):=Y(\mathrm{z})+\frac{t}{\rho(x_{0})}(x_{0}-y_{0}).
\]
Note that $G$ is a $C^{k,\alpha}$ function. Also note that we have
$G(0,0)=y_{0}$. Then similarly to the proof of Theorem \ref{thm:rho_is_C^2},
we can show that $\det DG(0,0)\ne0$, and by the inverse function
theorem, $G$ is invertible on an open set of the form $W\times(-h,h)$,
and it has a $C^{k,\alpha}$ inverse on $B_{r}(y_{0})$. Now we know
that $G:(\mathrm{z},t)\mapsto x$ has an inverse, denoted by $\mathrm{z}(x),t(x)$,
where $\mathrm{z}(\cdot),t(\cdot)$ are $C^{k,\alpha}$ functions
of $x$. Let $y:=Y(\mathrm{z}(x))$. Then we have 
\[
x=G(\mathrm{z}(x),t(x))=y+t(x)\,D\gamma^{\circ}(\nu_{0}).
\]
On the other hand, (\ref{eq:Parametrize_by_rho}) and (\ref{eq:Parametrize_by_-rho})
imply that 
\[
x=\hat{y}+\rho_{s}(x)\,D\gamma^{\circ}(\nu(\hat{y})),
\]
where $\hat{y}$ is one of the $\rho$-closest or $\bar{\rho}$-closest
points on $\partial U$ to $x$ (depending on whether $x\in\overline{U}$
or $x\in\R^{n}-\overline{U}$; note that when $x=\hat{y}\in\partial U$
the equation holds trivially). Then similarly to the proof of Theorem
\ref{thm:rho_is_C^2} we can show that 
\[
x=\hat{y}+\rho_{s}(x)\,D\gamma^{\circ}(\nu_{0})=Y(\hat{\mathrm{z}})+\rho_{s}(x)\,D\gamma^{\circ}(\nu_{0})
\]
for some $\hat{\mathrm{z}}\in W$. Hence $(\hat{\mathrm{z}},\rho_{s}(x))\in W\times(-h,h)$,
and we have $G(\hat{\mathrm{z}},\rho_{s}(x))=x$. Therefore due to
the invertibility of $G$ we must have 
\[
\hat{\mathrm{z}}=\mathrm{z}(x),\qquad\rho_{s}(x)=t(x).
\]
Hence $\rho_{s}$ is a $C^{k,\alpha}$ function on $B_{r}(y_{0})$.
Thus $\rho$ is a $C^{k,\alpha}$ function on $\overline{U}\cap B_{r}(y_{0})$.
In addition, it follows that every $x\in\overline{U}\cap B_{r}(y_{0})$
has a unique $\rho$-closest point on $\partial U$ given by $Y(\mathrm{z}(x))$,
because $\hat{y}=Y(\hat{\mathrm{z}})=Y(\mathrm{z}(x))$. 

Next, for $y\in Y(W)\cap B_{r}(y_{0})\subset\partial U$ consider
the line segment 
\[
t\mapsto y+tD\gamma^{\circ}(\nu(y))=y+tD\gamma^{\circ}(\nu_{0}),
\]
where $t\in(0,\tilde{h})$. If $\tilde{h}>0$ is small enough then
this segment lies inside $U\cap B_{r}(y_{0})$, since we know that
$\big\langle D\gamma^{\circ}(\nu),\nu\big\rangle=\gamma^{\circ}(\nu)>0$.
Now similarly to the last paragraph we can show that if $x$ belongs
to this segment, then $y$ is the unique $\rho$-closest point on
$\partial U$ to $x$. Hence $\rho$ is differentiable at $x$ and
similarly to the proof of Theorem \ref{thm:rho_is_C^2}
we can show that $D\rho(x)=\nu(y)/\gamma^{\circ}(\nu)$. Hence if
we let $x$ approach $y$ along this segment we get 
\[
D\rho(y)=\lim_{x\to y}D\rho(x)=\lim_{x\to y}\nu(y)/\gamma^{\circ}(\nu)=\nu(y)/\gamma^{\circ}(\nu),
\]
because $D\rho$ is continuous. In addition, when $k\ge2$, by the
continuity of $D^{2}\rho$ on $\overline{U}\cap B_{r}(y_{0})$ we
get 
\begin{align*}
D^{2}\rho(y)=\lim_{x\to y}D^{2}\rho(x) & =\lim_{x\to y}\frac{1}{\gamma^{\circ}(\nu)}(I-X^{T})D^{2}d(y)(I-X)\\
 & =\frac{1}{\gamma^{\circ}(\nu)}(I-X^{T})D^{2}d(y)(I-X),
\end{align*}
where the formula for $D^{2}\rho(x)$ can be obtained similarly to
the proof of Theorem \ref{thm:rho_is_C^2}.
\end{proof}

\section{Examples}

\begin{example}

Let us start with a two-dimensional example and compute the distance to the parabola $x_2=x_1^2$ with respect to the maximum norm $\gamma(x)=|x|_{\infty}=\max \{ |x_{1}| , |x_{2}| \}$. In this
case $K$ is the square $\{x\in\mathbb{R}^{2}:-1\le x_{1} , x_2\le1 \}$, and we will denote $\rho$ by $d_\infty$. Applying Lemma \ref{lem:K_ball_touch_bdry} we see that for $x$ above the parabola its $d_\infty$-closest point on the parabola must be a lower corner of the square $K_x$, since by \eqref{eq:-nu_in_Kx} the outward normal vector $-\nu(y)$ to the parabola must belong to the normal cone $N(K_x,y)$, and the lower corners of $K_x$ are the only points whose corresponding normal cones contain an outward normal to the parabola. Also note that no square $K_x$ above the parabola can touch the parabola only at $y=0$ due to its strict convexity. Similarly, for $x$ below the parabola, its $d_\infty$-closest point on the parabola must be an upper corner of the square $K_x$ or belong to its upper side; see Figure \ref{fig:parabola}.

Now for $x$ above the parabola (i.e.\@ $x_2 > x_1 ^2$), the lower corners of $K_x$ are 
\[
\big( x_1 \pm d_\infty (x), \, x_2 - d_\infty (x) \big).
\]
Hence we must have 
\[
 x_2 - d_\infty (x) = \big( x_1 \pm d_\infty (x) \big)^2,
\]
where $+$ is chosen when $x_1 >0$ and $-$ is chosen when $x_1 <0$ (when $x_1 =0$ we can choose both). 
Therefore 
\[
(d_\infty  (x))^2 +(1 + 2|x_1|) d_\infty (x) + x_1^2 -x_2 =0.
\]
Thus we get 
\[ d_\infty (x) = \frac{1}{2} \Bigl(  -1 - 2|x_1| + \sqrt{ 4 |x_1| + 4 x_2 + 1} \Big)  =   \sqrt{  |x_1| +  x_2 + \frac{1}{4}}  - |x_1| - \frac{1}{2}.
\]
On the other hand, for $x$ below the parabola (i.e.\@ $x_2 < x_1 ^2$), the upper corners of $K_x$ are 
\[
\big( x_1 \mp d_\infty (x) , \, x_2 + d_\infty (x) \big).
\]
But in this case $y$ can also be on the upper side of $K_x$. Let us characterize those points $x$ for which this is the case. Note that for such $y$ the inward normal to the parabola (viewed from its exterior) is $\nu (y)=(0,-1)$. Hence we must have $y=(0,0)$. Then by \eqref{eq:K-normal} and \eqref{eq:dg} we obtain  
\[
\frac{x-y\;\;}{|x-y|_\infty} \in \partial \gamma^\circ (\nu(y)) = \partial \gamma^\circ ( (0,-1))  = N(K^\circ , (0,-1) ) \cap \partial K, 
\]
where $\gamma^\circ$ is the 1-norm $\gamma^{\circ}(x)=|x|_{1}=|x_1|+|x_2|$ and $K^\circ$ is the rhombus $\{|x_1|+|x_2| \le 1 \}$. But $N(K^\circ , (0,-1) ) \cap \partial K$ is just the lower side of $K$. So we must have 
\[
\frac{x-y\;\;}{|x-y|_\infty} \in \{ (t,-1) : -1 \le t \le 1 \}.
\]
Since $y=(0,0)$ we get 
\[
\frac{(x_1,x_2)}{\max \{ |x_1|,|x_2| \}} \in \{ (t,-1) : -1 \le t \le 1 \},
\]
which holds if and only if $|x_1| \le |x_2|$ and $x_2 <0$. In this region we have 
\[
d_\infty (x) = |x-y|_\infty = |x|_\infty = |x_2| = -x_2.
\]
Below the parabola and outside this region we can compute $d_\infty (x)$ as before to get 
\[ d_\infty (x)  = -  \sqrt{  |x_1| +  x_2 + \frac{1}{4}}  + |x_1| + \frac{1}{2}.
\]
Note that the smaller positive root of the corresponding quadratic equation gives the correct value for $d_\infty (x)$, since on the parabola we must have $d_\infty =0$. Therefore we have 
\begin{equation}
d_{\infty}(x)=\begin{cases}
\sqrt{  |x_1| +  x_2 + \frac{1}{4}}  - |x_1| - \frac{1}{2} & \textrm{if }x_2 > x_1^2,\\
\\ -x_2 & \textrm{if } |x_1| \le -x_2,
\\ \\
 -\sqrt{  |x_1| +  x_2 + \frac{1}{4}}  + |x_1| + \frac{1}{2} & \textrm{otherwise}.
\end{cases}\label{eq:d_infty-parabola}
\end{equation}
Note that there is no point satisfying $|x_1|\le -x_2$ when $x_2 >0$, so, in the above formula, we do not need to explicitly require $x_2 \le 0$ in the description of the second region.

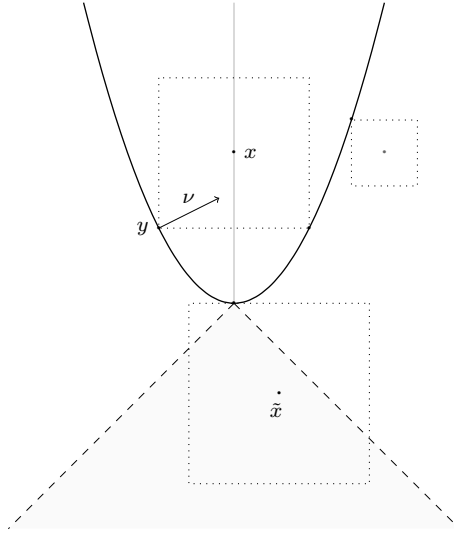
\begin{figure}
\begin{tikzpicture}

\draw[black, line width = 0.5pt]   plot[smooth,domain=-2:2] (\x, {(\x)^2});

\draw[color=black,dotted] (-1,3) rectangle (1,1);
\draw (0,2) node[anchor=center] {$\cdot$} node[anchor=west] {$\scriptstyle{x}\;$}; 
\draw (-1,1) node[anchor=center] {$\cdot$} node[anchor=east] {$\scriptstyle{y}$}; 
\draw (1,1) node[anchor=center] {$\cdot$};

\draw [line width=0.1pt, ->] (-1,1) -- (-0.2,1.4);
\draw (-0.6,1.2) node[anchor=south] {$\scriptstyle{\nu}$};

\draw[color=black,dotted] (1.5616,2.4384) rectangle (2.4384,1.5616);
\draw (2,2) node[anchor=center, opacity=0.6] {$\cdot$}; 
\draw (1.5616,2.4384) node[anchor=center] {$\cdot$};

\draw[color=black,dotted] (1.8,0) rectangle (-0.6,-2.4);
\draw (0.6,-1.2) node[anchor=center] {$\cdot$} node[anchor=north] {$\scriptstyle{\tilde{x}}\;$}; 
\draw (0,0) node[anchor=center] {$\cdot$};

\draw [line width=0.1pt, opacity = 0.3] (0,0) -- (0,4);

\draw [line width=0.2pt, dashed] (0,0) -- (3,-3);
\draw [line width=0.2pt, dashed] (0,0) -- (-3,-3);
\fill [black,fill opacity=0.02] (0,0) -- (3,-3) -- (-3,-3) -- cycle;

\end{tikzpicture}

\caption{\protect\label{fig:parabola} The different regions in the plane over
which $d_{\infty}(x)$ has different formulas. 
The point $y$ is the $d_{\infty}$-closest point to $x$, and $\nu$ is the normal to the parabola. Note that the intersection of $K_{\tilde{x}}$ with the parabola is
at a non-vertex point of $\partial K_{\tilde{x}}$.}
\end{figure}

Let us also compute the derivative of the corresponding signed distance function (which is equal to $-d_\infty$ on the exterior of the parabola, i.e.\@ below it) 
\begin{equation}
D(d_{\infty})_s(x)=\begin{cases}
\Big( \frac{x_1}{2|x_1|}\big( |x_1| +  x_2 + \frac{1}{4}\big)^{-\frac{1}{2}}  - \frac{x_1}{|x_1|} ,\;  \frac{1}{2}\big( |x_1| +  x_2 + \frac{1}{4}\big)^{-\frac{1}{2}}  \Big)
& \textrm{if }x_2 > x_1^2 \textrm{ and } x_1 \ne 0,\\
\\ (0,1) & \textrm{if } |x_1| \le -x_2,
\\ \\
 \Big( \frac{x_1}{2|x_1|}\big( |x_1| +  x_2 + \frac{1}{4}\big)^{-\frac{1}{2}}  - \frac{x_1}{|x_1|} ,\;  \frac{1}{2}\big( |x_1| +  x_2 + \frac{1}{4}\big)^{-\frac{1}{2}}  \Big) & \textrm{if } x_2 \le x_1 ^2 \textrm{ and } |x_1|> -x_2 .
\end{cases}\label{eq:Dd_infty-parabola}
\end{equation}
Note that the signed distance function is differentiable everywhere except when $x_1 =0$ and $x_2 > 0$, corresponding to the points with more than one closest point on the parabola, which is compatible with the case of distance functions corresponding to smooth convex sets. However, unlike the smooth case, it may happen that the distance function is differentiable at points with more than one closest point on the boundary, as we will see in Example \ref{exterior-2-balls}. In addition, unlike the distance functions corresponding to smooth convex sets, here the set of nondifferentiability touches the boundary. 
Furthermore, note that although the first derivative is continuous across the boundary of the region $|x_1|\le -x_2$,  the second derivative does not exist there. 
\end{example}

\begin{example}
Next, let us compute the distance to the unit sphere $\partial B_{1}(0)$
with respect to the gauge function $\gamma$ corresponding to a polytope
$K$. We assume that the vertices of $K$ are equidistant from the
origin, i.e.\@ $K$ is inscribed in the sphere $\partial B_{r}(0)$
for some $r>0$. 

Let $x\in B_{1}$. Consider the convex set $K_{x}=x-\rho(x)K$. Then
by Lemma \ref{lem:K_ball_touch_bdry} we
know that $\mathrm{int}(K_{x})\subset B_{1}$, and $\partial K_{x}\cap\partial B_{1}$
is the set of $\rho$-closest points on $\partial B_{1}$ to $x$.
Let $y\in\partial K_{x}\cap\partial B_{1}$. Then $y$ must be a vertex
of $K_{x}$. Otherwise, a line segment passing through $y$ lies in
$\partial K_{x}$, and this line segment cannot touch $\partial B_{1}$
and stay inside $\overline{B}_{1}$ at the same time, due to the strict convexity of
$B_{1}$; so the line segment will intersect the exterior of the unit
ball $B_{1}$, which is a contradiction. 

Now we have $y=x-\rho(x)z$ for some $z\in\partial K$. Note that
$z$ is a vertex of $K$; so we have $z\in\partial B_{r}$ by our
assumption. Hence 
\begin{equation}
1=|y|^{2}=|x-\rho(x)z|^{2}=|x|^{2}-2\rho(x)\langle x,z\rangle+\rho(x)^{2}|z|^{2}=|x|^{2}-2\rho(x)\langle x,z\rangle+\rho(x)^{2}r^{2}.\label{eq:1_in_examp_1}
\end{equation}
On the other hand, for any other vertex $\tilde{z}\in\partial K$
we have $x-\rho(x)\tilde{z}\in K_{x}\subset B_{1}$; thus 
\[
1\ge|x-\rho(x)\tilde{z}|^{2}=|x|^{2}-2\rho(x)\langle x,\tilde{z}\rangle+\rho(x)^{2}|\tilde{z}|^{2}=|x|^{2}-2\rho(x)\langle x,\tilde{z}\rangle+\rho(x)^{2}r^{2}.
\]
Therefore we must have $\langle-x,z\rangle\ge\langle-x,\tilde{z}\rangle$.
But, since all points of the convex polytope $K$ are convex combinations
of its vertices, we have 
\[
\langle-x,z\rangle\ge\sup_{\tilde{x}\in K}\langle-x,\tilde{x}\rangle=\gamma^{\circ}(-x),
\]
where the last equality follows from (\ref{eq:g0_is_h}). Hence $-\langle x,z\rangle=\langle-x,z\rangle=\gamma^{\circ}(-x)$.
Plugging this in (\ref{eq:1_in_examp_1}) we obtain 
\begin{align*}
|x|^{2}+2\rho(x)\gamma^{\circ}(-x) & +\rho(x)^{2}r^{2}=1\\
 & \implies\rho(x)=\frac{1}{r^{2}}\Bigl(-\gamma^{\circ}(-x)+\sqrt{\gamma^{\circ}(-x)^{2}+r^{2}-r^{2}|x|^{2}}\,\Big).
\end{align*}
Note that the other root of the quadratic equation is negative and
cannot be equal to $\rho(x)$. It is easy to see that this formula
gives us $\rho(x)=0$ when $x$ is on the unit sphere ($|x|=1$). 

Next suppose $x$ is outside the unit ball, i.e.\@ $|x|>1$. Let $y\in\partial K_{x}\cap\partial B_{1}$
be a $\rho$-closest point on $\partial B_{1}$ to $x$. Then we have
$y=x-\rho(x)z$ for some $z\in\partial K$. But in this case $z$
is not necessarily a vertex of $K$. However, if we further assume
that $z$ is a vertex of $K$, then similarly to the above we have
\[
1=|y|^{2}=|x-\rho(x)z|^{2}=|x|^{2}-2\rho(x)\langle x,z\rangle+\rho(x)^{2}r^{2}.
\]
But this time for any other vertex $\tilde{z}\in\partial K$ we have
$x-\rho(x)\tilde{z}\in K_{x}\subset\R^{n}-B_{1}$; thus 
\[
1\le|x-\rho(x)\tilde{z}|^{2}=|x|^{2}-2\rho(x)\langle x,\tilde{z}\rangle+\rho(x)^{2}r^{2}.
\]
So $\langle x,z\rangle\ge\langle x,\tilde{z}\rangle$, and hence,
similarly to the above we get $\langle x,z\rangle=\gamma^{\circ}(x)$.
Therefore we obtain 
\begin{align*}
|x|^{2}-2\rho(x)\gamma^{\circ}(x) & +\rho(x)^{2}r^{2}=1\\
 & \implies\rho(x)=\frac{1}{r^{2}}\Big(\gamma^{\circ}(x)-\sqrt{\gamma^{\circ}(x)^{2}+r^{2}-r^{2}|x|^{2}}\,\Big).
\end{align*}
Note that we have taken $\rho(x)$ to be the smaller root of the quadratic
equation. (Because for the larger root $s=\frac{1}{r^{2}}\big[\gamma^{\circ}(x)+\sqrt{\gamma^{\circ}(x)^{2}+r^{2}-r^{2}|x|^{2}}\,\big]$,
the point $y=x-\rho(x)z=x-s(\frac{\rho(x)}{s}z)$ is on $\partial B_{1}$
and is also in the interior of $x-sK$, since $\frac{\rho(x)}{s}z\in\mathrm{int}(K)$.
So $s$ cannot be $\rho(x)$ due to Lemma \ref{lem:K_ball_touch_bdry}.)
Notice that for this formula to give a real value for $\rho(x)$
we must have 
\begin{equation}
\gamma^{\circ}(x)^{2}+r^{2}\ge r^{2}|x|^{2}.\label{eq:obstr_y_is_vertex}
\end{equation}
Thus this is an obstruction for $y$ to be a vertex of $K_{x}$, or
in other words, an obstruction for $x$ to have a $\rho$-closest
point which is a vertex of $K_{x}$. (As we will see in the next
example, this obstruction is not sharp, and the actual region of those
points $x$ whose $\rho$-closest point is a vertex of $K_{x}$ can
be smaller.) 

Note that when $x$ is outside the unit ball the $\rho$-closest point
to $x$ must be unique. Because if $\partial K_{x}$ touches $\partial B_{1}$
at two points $y,\tilde{y}$, then the line segment $]y,\tilde{y}[$
will be inside $K_{x}$ and $B_{1}$ due to their convexity, which
is in contradiction with Lemma \ref{lem:K_ball_touch_bdry}.
However, as we have already noted, when $x$ is outside the unit ball,
its $\rho$-closest point $y\in\partial B_{1}$ is not necessarily
a vertex of $K_{x}$. 

So suppose $y$ is not a vertex of $K_{x}$, or equivalently, $z=\frac{x-y}{\rho(x)}$
is not a vertex of $K$. (Note that the following arguments, with
obvious modifications, also work when $z$ is a vertex, although
we have already dealt with this case using a different approach.)
Then by (\ref{eq:K-normal}) we have (noting that the inward unit
normal at $y$ to $\partial B_{1}$, considered as the boundary of
$\mathbb{R}^{n}-B_{1}$, is $y$) 
\[
z=\frac{x-y}{\gamma(x-y)}\in\partial\gamma^{\circ}(\nu(y))=\partial\gamma^{\circ}(y).
\]
Thus by Lemma \ref{lem:x_in_dg(v)} and (\ref{eq:dg})
we get 
\[
\tilde{y}:=\frac{y}{\gamma^{\circ}(y)}\in\partial\gamma(z)=N(K,z)\cap\partial K^{\circ}.
\]
Note that $\tilde{y}$ is a singular point of $\partial K^{\circ}$
by Lemma \ref{lem:normal_to_face}, since
$z$ is not a vertex of $K$. Also note that $|\tilde{y}|=|y|/\gamma^{\circ}(y)=1/\gamma^{\circ}(y)$.
So we obtain $y=\tilde{y}/|\tilde{y}|$. In other words, $y$ belongs
to the image of the singular points of $\partial K^{\circ}$ under
the map 
\[
w\mapsto\frac{w}{|w|}.
\]

Conversely, suppose $y=\tilde{y}/|\tilde{y}|$ for some $\tilde{y}\in\partial K^{\circ}$.
Let us show that $y$ is the $\rho$-closest point on $\partial B_{1}$
to some points in $\mathbb{R}^{n}-B_{1}$. In fact, this is the case
for the points $x=y+tz$ with $t>0$ and $z\in N(K^{\circ},\tilde{y})\cap\partial K$.
We claim that 
\[
L:=x-tK\subset\mathbb{R}^{n}-B_{1}.
\]
To see this note that by (\ref{eq:v_in_N(x)}) we have $\tilde{y}\in N(K,z)$
and thus for every $w\in K$ we have $\langle w-z,\tilde{y}\rangle\le0$;
hence for $w\ne z$ we get 
\begin{align*}
|x-tw|^{2}=|y+t(z-w)|^{2} & =|y|^{2}+2t\langle y,z-w\rangle+t^{2}|z-w|^{2}\\
 & =1-2t|\tilde{y}|\langle\tilde{y},w-z\rangle+t^{2}|z-w|^{2}>1.
\end{align*}
In addition, for $w=z$ we have $y=x-tz\in\partial L\cap\partial B_{1}$.
So by Lemma \ref{lem:K_ball_touch_bdry}
we must have $t=\rho(x)$ and $L=K_{x}$. In particular, $y$ is the
$\rho$-closest point on $\partial B_{1}$ to $x=y+tz$ for every
$t>0$ and $z\in N(K^{\circ},\tilde{y})\cap\partial K$. %

Therefore, putting all these together, for $x\in\mathbb{R}^{n}$ we
have 
\[
\rho(x)=d_{K}(x,\partial B_{1})=\begin{cases}
\frac{1}{r^{2}}\bigl(-\gamma^{\circ}(-x)+\sqrt{\gamma^{\circ}(-x)^{2}+r^{2}-r^{2}|x|^{2}}\,\big) & \textrm{if }|x|\le1,\\
\\\gamma(x-y) & \begin{matrix*}[l] \textrm{if }|x|>1\textrm{ and }x-y\in N(K^{\circ},\tilde{y})\\
  \,\textrm{where }|y|=1\textrm{ and }\tilde{y}=y/\gamma^{\circ}(y)\\
  \,\textrm{is a singular point of }\partial K^{\circ}, \end{matrix*}\\ \\
\frac{1}{r^{2}}\big(\gamma^{\circ}(x)-\sqrt{\gamma^{\circ}(x)^{2}+r^{2}-r^{2}|x|^{2}}\,\big) & \textrm{otherwise}.
\end{cases}
\]
Interestingly, when $K$  is not symmetric with respect to the origin, this formula shows that the signed distance function to $\partial B_1$ with respect to $\gamma$ may not be even once differentiable on $\partial B_1$. This is in contrast to the case of distance functions with respect to smooth norms. Note that outside a set with $(n-1)$-Hausdorff measure zero, we can use (the negative of) the last case of the above formula to evaluate the derivative of the signed distance function on $\partial B_1 $.    
\end{example}
\begin{example}\label{exa:d_infty}
Let us consider a special case of the above example, and compute the
distance to the unit sphere $\partial B_{1}(0)$ with respect to the
maximum norm $\gamma(x)=|x|_{\infty}=\max_{i\le n}|x_{i}|$. In this
case $K$ is the cube $\{x\in\mathbb{R}^{n}:-1\le x_{i}\le1\textrm{ for }i=1,\dots,n\}$.
Note that the cube $K$ is inscribed in the sphere $\partial B_{{\scriptscriptstyle \sqrt{n}}}(0)$.
We also have $\gamma^{\circ}(x)=|x|_{1}=\sum_{i=1}^{n}|x_{i}|$, and
\[
K^{\circ}=\{x\in\mathbb{R}^{n}:|x_{1}|+\dots+|x_{n}|\le1\}.
\]
The $\partial K^{\circ}$ consists of $2^{n}$ facets (the same number
as the vertices of $K$) determined by 
\[
\epsilon_{1}x_{1}+\dots+\epsilon_{n}x_{n}=1,\qquad\epsilon_{i}x_{i}\ge0
\]
for each choice of $\epsilon_{i}\in\{-1,1\}$%
. Note that the vertex $(\epsilon_{1},\dots,\epsilon_{n})$ of $K$
is an outer normal vector to the above facet. The relative interior
of these facets are easily seen to be 
\[
\epsilon_{1}x_{1}+\dots+\epsilon_{n}x_{n}=1,\qquad\epsilon_{i}x_{i}>0.
\]
Hence their relative boundaries, which form the set of singular points
of $\partial K^{\circ}$, are%
{} 
\[
\sum_{i\ne j}\epsilon_{i}x_{i}=1,\;x_{j}=0\quad\textrm{for some }j\in\{1,\dots,n\}.
\]
Note that the union of these sets is the union of the intersections
of $\partial K^{\circ}$ with the hyperplanes $\{x_{j}=0\}$. So the
image of the set of singular points of $\partial K^{\circ}$ under
the map $w\mapsto w/|w|$ is the union of the intersections of $\partial B_{1}$
with the hyperplanes $\{x_{j}=0\}$. 

Now let $z$ be a singular point of $\partial K^{\circ}$ of the form
\begin{align}
 & z_{j}=0\quad\textrm{for }j\in J=\{j_{1},\dots,j_{m}\},\nonumber \\
 & \sum_{i\notin J}\epsilon_{i}z_{i}=1,\quad\epsilon_{i}z_{i}>0\quad\textrm{for }i\notin J.\label{eq:z_sing}
\end{align}
Note that $\epsilon_{i}=z_{i}/|z_{i}|$ is the sign of $z_{i}$. Then
$z$ is in the intersection of the $2^{m}$ facets of $\partial K^{\circ}$
determined by 
\[
\sum_{i\notin J}\epsilon_{i}x_{i}+\sum_{j\in J}\eta_{j}x_{j}=1,\qquad\epsilon_{i}x_{i},\eta_{j}x_{j}\ge0
\]
for each choice of $\eta_{j}\in\{-1,1\}$. By Lemma \ref{lem:normal_to_face}
we have 
\begin{align*}
N(K^{\circ},z)\cap\partial K & =\mathrm{conv}\,\{(\eta_{1},\dots,\eta_{n}):\eta_{j}=\pm1\textrm{ for }j\in J,\textrm{ and }\eta_{i}=\epsilon_{i}\textrm{ for }i\notin J\}\\
 & =\{v:-1\le v_{j}\le1\textrm{ for }j\in J,\textrm{ and }v_{i}=z_{i}/|z_{i}|\textrm{ for }i\notin J\}.
\end{align*}
As shown in the last example, we know that if $x-z/|z|\in N(K^{\circ},z)$
then 
\[
d_{\infty}(x,\partial B_{1})=\rho(x)=\gamma(x-z/|z|)=\bigl|x-z/|z|\bigr|_{\infty}.
\]
Therefore if $x=\frac{z}{|z|}+tv$ where $t\ge0$, $|v_{j}|\le1$
for $j\in J$, and $v_{i}=\frac{z_{i}}{|z_{i}|}$ for $i\notin J$,
then 
\[
d_{\infty}(x,\partial B_{1})=|tv|_{\infty}=t|v|_{\infty}=t.
\]
Note that componentwise we have 
\[
x_{k}=\frac{z_{k}}{|z|}+tv_{k}=\begin{cases}
tv_{k} & k\in J,\\
\frac{z_{k}}{|z|}+t\frac{z_{k}}{|z_{k}|} & k\notin J.
\end{cases}
\]
So for $i\notin J$ we have 
\[
t=1\cdot t=\mathrm{sgn}(z_{i})\frac{z_{i}}{|z_{i}|} t=\mathrm{sgn}(z_{i})\Bigl(t\frac{z_{i}}{|z_{i}|}\Bigr)=\mathrm{sgn}(z_{i})\Bigl(x_{i}-\frac{z_{i}}{|z|}\Bigr).
\]
And for $j\in J$ we have 
\[
|x_{j}|=|tv_{j}|=t|v_{j}|\le t=\mathrm{sgn}(z_{i})\Bigl(x_{i}-\frac{z_{i}}{|z|}\Bigr).
\]
Hence, for any singular point $z\in\partial K^{\circ}$ satisfying
(\ref{eq:z_sing}), over the set 
\begin{align*}
 & \{x:|x|>1,\\
 & \qquad\textrm{and the value of }\mathrm{sgn}(z_{i})\bigl(x_{i}-z_{i}/|z|\bigr)\textrm{ is independent of }i\textrm{ for }i\notin J,\\
 & \qquad\textrm{and }|x_{j}|\le\mathrm{sgn}(z_{i})\bigl(x_{i}-z_{i}/|z|\bigr)\textrm{ for }j\in J\}
\end{align*}
we have $d_{\infty}(x,\partial B_{1})=\mathrm{sgn}(z_{i})\bigl(x_{i}-z_{i}/|z|\bigr)$.

Therefore, in light of the results of the previous example, for $x\in\mathbb{R}^{n}$
we have 
\begin{equation}
d_{\infty}(x,\partial B_{1})=\begin{cases}
\frac{1}{n}\Bigl[-|x|_{1}+\Big(|x|_{1}^{2}+n-n|x|^{2}\Big)^{1/2}\Bigr] & \textrm{if }|x|\le1,\\
\\\mathrm{sgn}(z_{i})\bigl(x_{i}-z_{i}/|z|\bigr) & \begin{matrix*}[l] \textrm{if }|x|>1\textrm{ and }x-z/|z|\in N(K^{\circ},z)\\
  \,\textrm{where }z\textrm{ is a singular point of }\partial K^{\circ}\\
  \,\textrm{satisfying }(\ref{eq:z_sing}), \end{matrix*} \\ \\
\frac{1}{n}\Bigl[|x|_{1}-\Big(|x|_{1}^{2}+n-n|x|^{2}\Big)^{1/2}\Bigr] & \textrm{otherwise}.
\end{cases}\label{eq:d_infty}
\end{equation}

\begin{figure}
\begin{tikzpicture}

\draw [line width=0.3pt,fill=black,fill opacity=0.07] (0.,0.) circle (1.cm); 

\draw[color=black,dotted] (-2.5,-0.5) rectangle (-1,1);
\draw (-1.75,0.25) node[anchor=center] {$\cdot$} node[anchor=north] {$\scriptstyle{x}\;$}; 
\draw (-1,0) node[anchor=center] {$\cdot$} node[anchor=west] {$\scriptstyle{y}$}; 

\draw [line width=0.2pt, dashed] (1,0) -- (4.5,3.5);
\draw (3,0) node[anchor=center] {$\scriptstyle{
d_\infty (x) = |x_1|-1}$}; 
\draw [line width=0.2pt, dashed] (1,0) -- (4.5,-3.5);
\fill [black,fill opacity=0.03] (1,0) -- (4.5,3.5) -- (4.5,-3.5) -- cycle;

\draw [line width=0.2pt, dashed] (-1,0) -- (-4.5,3.5);
\draw [line width=0.2pt, dashed] (-1,0) -- (-4.5,-3.5);
\fill [black,fill opacity=0.03] (-1,0) -- (-4.5,3.5) -- (-4.5,-3.5) -- cycle;

\draw [line width=0.2pt, dashed] (0,1) -- (3.5,4.5);
\draw (0,3) node[anchor=center] {$\scriptstyle{
d_\infty (x) = |x_2|-1}$}; 
\draw [line width=0.2pt, dashed] (0,1) -- (-3.5,4.5);
\fill [black,fill opacity=0.01] (0,1) -- (3.5,4.5) -- (-3.5,4.5) -- cycle;

\draw [line width=0.2pt, dashed] (0,-1) -- (3.5,-4.5);
\draw [line width=0.2pt, dashed] (0,-1) -- (-3.5,-4.5);
\fill [black,fill opacity=0.01] (0,-1) -- (3.5,-4.5) -- (-3.5,-4.5) -- cycle;

\end{tikzpicture}

\caption{\protect\label{fig:2}The different regions in the plane over
which $d_{\infty}(x,\partial B_{1})$ has different formulas. In the
unshaded region and inside the unit disk, $d_{\infty}$ is given by
the last and first cases of the formula (\ref{eq:d_infty}) respectively.
The point $y$ is the $d_{\infty}$-closest point to $x$. Note that
in this case the intersection of $K_{x}$ with $\partial B_{1}$ is
at a non-vertex point of $\partial K_{x}$.}
\end{figure}
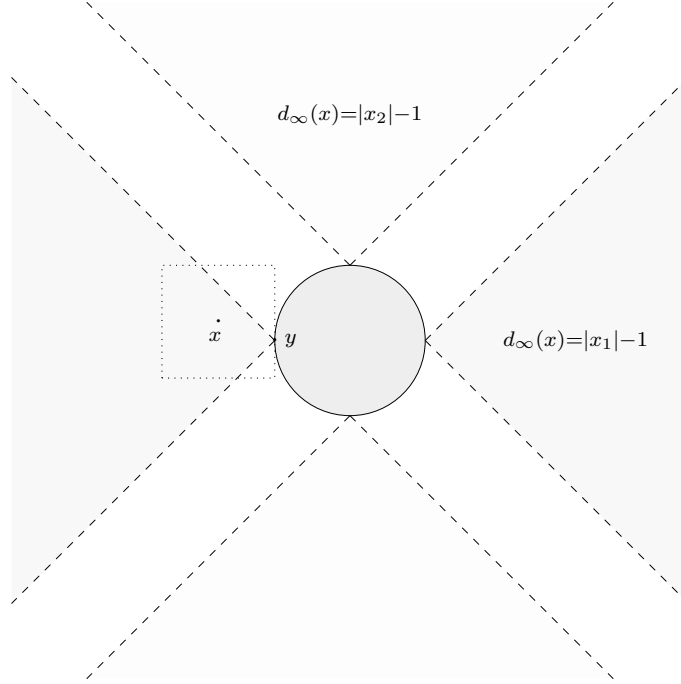

Let us consider the case of $n=2$. In this case $z$ can be one of
the four points $(\pm1,0)$ and $(0,\pm1)$. When $z=(\pm1,0)$, over
the corresponding regions 
\[
\{|x|>1,\;|x_{2}|\le(\pm1)(x_{1}\mp1)\}
\]
we have $d_{\infty}(x,\partial B_{1})=(\pm1)(x_{1}\mp1)=\pm x_{1}-1=|x_{1}|-1$.
Similarly, when $z=(0,\pm1)$, we have $d_{\infty}(x,\partial B_{1})=|x_{2}|-1$
over the corresponding regions. So overall for $n=2$ we have (see
Figure \ref{fig:2}) 
\begin{align*}
d_{\infty}(x,\partial B_{1}) & =\begin{cases}
\frac{1}{2}\Bigl[-|x|_{1}+\Big(|x|_{1}^{2}+2-2|x|^{2}\Big)^{1/2}\Bigr] & \textrm{if }|x|\le1,\\
|x_{1}|-1 & \textrm{if }|x|>1\textrm{ and }|x_{2}|\le|x_{1}|-1,\\
|x_{2}|-1 & \textrm{if }|x|>1\textrm{ and }|x_{1}|\le|x_{2}|-1,\\
\frac{1}{2}\Bigl[|x|_{1}-\Big(|x|_{1}^{2}+2-2|x|^{2}\Big)^{1/2}\Bigr] & \textrm{otherwise},
\end{cases} \allowdisplaybreaks \\
 & =\begin{cases}
\frac{1}{2}\Bigl[-|x_{1}|-|x_{2}|+\Big(2-(|x_{1}|-|x_{2}|)^{2}\Big)^{1/2}\Bigr] & \textrm{if }|x|\le1,\\
|x_{1}|-1 & \textrm{if }|x_{2}|\le|x_{1}|-1,\\
|x_{2}|-1 & \textrm{if }|x_{1}|\le|x_{2}|-1,\\
\frac{1}{2}\Bigl[|x_{1}|+|x_{2}|-\Big(2-(|x_{1}|-|x_{2}|)^{2}\Big)^{1/2}\Bigr] & \textrm{otherwise}.
\end{cases}
\end{align*}
Note that here the obstruction formula (\ref{eq:obstr_y_is_vertex})
for $x$ to have a $\rho$-closest point that is a vertex of $K_{x}$
is not sharp, and the actual region of those points (i.e.\@ the set
of points for which the last case of the above formula holds) is smaller.
Since the formula requires that 
\[
|x|_{1}^{2}+2\ge2|x|^{2}\iff(|x_{1}|-|x_{2}|)^{2}\le2,
\]
while we actually have $(|x_{1}|-|x_{2}|)^{2}\le1$ when $|x|>1$.

Next let us consider the case of $n=3$. In this case $z$ has one
0 component and the other two components are of the form $\pm t$
and $\pm(1-t)$ for some $0\le t\le1$. Consider for example $z=(0,t,1-t)$.
Then the set of points $x$ having $z/|z|$ as their $d_{\infty}$-closest
point is determined by 
\[
x_{2}-\frac{t}{\sqrt{t^{2}+(1-t)^{2}}}=x_{3}-\frac{1-t}{\sqrt{t^{2}+(1-t)^{2}}}=d_{\infty}(x),
\]
and $|x_{1}|\le d_{\infty}(x)$. Hence 
\[
x_{2}-x_{3}=\frac{2t-1}{\sqrt{t^{2}+(1-t)^{2}}}\implies(x_{2}-x_{3})^{2}=\frac{(2t-1)^{2}}{t^{2}+(1-t)^{2}}=2-\frac{1}{t^{2}+(1-t)^{2}}.
\]
Thus $t^{2}+(1-t)^{2}=\frac{1}{2-(x_{2}-x_{3})^{2}}$, and therefore
\[
2t-1=(x_{2}-x_{3})\sqrt{t^{2}+(1-t)^{2}}=\frac{x_{2}-x_{3}}{\sqrt{2-(x_{2}-x_{3})^{2}}}.
\]
So we get 
\begin{align*}
d_{\infty}(x) & =x_{2}-\frac{t}{\sqrt{t^{2}+(1-t)^{2}}}=x_{2}-\frac{\frac{1}{2}+\frac{x_{2}-x_{3}}{2\sqrt{2-(x_{2}-x_{3})^{2}}}}{\frac{1}{\sqrt{2-(x_{2}-x_{3})^{2}}}}\\
 & =x_{2}-\frac{1}{2}\big(\sqrt{2-(x_{2}-x_{3})^{2}}+x_{2}-x_{3}\big)=\frac{1}{2}\big(x_{2}+x_{3}-\sqrt{2-(x_{2}-x_{3})^{2}}\big).
\end{align*}
This is exactly the two-dimensional formula for $d_{\infty}(x,\partial B_{1})$
in the first quadrant of $(x_{2},x_{3})$-plane. 

In fact, the above observation is a manifestation
of a general property of $d_{\infty}(x,\partial B_{1})$. Namely, for a point $x$
in the hyperplane $\{x_{k}=0\}$ that satisfies $|x|>1$ we have 
\[
d_{\infty}^{n}(x,\partial B_{1})=d_{\infty}^{n-1}(\hat{x},\partial\hat{B}_{1}),
\]
where $\hat{x}$ is the projection of $x$ on the hyperplane, $\hat{B}_{1}$
is the unit disk in that hyperplane, and the superscript in $d_{\infty}$
denotes the dimension. This can be easily seen by applying Lemma \ref{lem:K_ball_touch_bdry}, because the intersection of $x-tK$ with the hyperplane is just $\hat{x} -t\hat{K}$, where $\hat{K}$ is the unit "ball" with respect to the maximum norm in $n-1$ dimensions.\footnote{This is a very particular property of the maximum norm and these hyperplanes. Similar conclusions cannot be made for intersections with arbitrary hyperplanes, or for arbitrary $K$.} And if $x-tK$ only touches $\partial B_1$ at one of its boundary points, the same is true about $\hat{x} -t\hat{K}$ and $\partial\hat{B}_{1}$ (which intersect at exactly the same point). This shows that the above property not only holds for points $x$ in the hyperplane, but also for those points $x$ for which $K_x = x - d_\infty (x) K$ intersects the hyperplane, which happens if and only if $|x_k| \le d_\infty (x)$. Because, in this case too, the unique intersection point of $K_x$ and $\partial B_1$, denoted by $y$, must lie in the hyperplane. The reason is that projection on the hyperplane $\{x_{k}=0\}$ does not increase the maximum norm, so $|\hat{y}-\hat{x}|_\infty \le d_\infty (x)$. Hence $|\hat{y}-x|_\infty \le d_\infty (x)$ as $|x_k| \le d_\infty (x)$; thus, noting that $|\hat{y}|\le 1$, we must have $\hat{y}=y$ due to the uniqueness of the intersection point for $|x|>1$. Notice that these points $x$ are exactly the points in the exterior of $B_1$ whose $d_\infty$-closest point on $\partial B_1$ corresponds to a singular point of $\partial K^\circ$. 

We can also iterate the above formula to extend it to higher codimensions. Consequently, when $K_x$ intersects the subspace $\{ x_{k_1}=\dots = x_{k_m}=0 \}$ we have 
\[ 
d_{\infty}^{n}(x,\partial B_{1})=d_{\infty}^{n-m}(\hat{x},\partial\hat{B}_{1}),
\]
where here $\hat{x}$ is the projection of $x$ on the subspace and $\hat{B}_{1}$
is the unit disk in that subspace. Combining these observations with formula \eqref{eq:d_infty} we obtain 
\begin{equation}
d_{\infty}(x,\partial B_{1})=\begin{cases}
\frac{1}{n}\Bigl[-|x|_{1}+\Big(|x|_{1}^{2}+n-n|x|^{2}\Big)^{1/2}\Bigr] & \textrm{if }|x|\le1,\\
\\\frac{1}{n-|J|}\Bigl[|\hat{x}|_{1} - \Big(|\hat{x}|_{1}^{2}+n-|J|-(n-|J|)|\hat{x}|^{2}\Big)^{1/2}\Bigr]  & \begin{matrix*}[l] \textrm{if }|x|>1\textrm{ and } \\ \,|x_j|\le d_\infty (x) \textrm{ for } j \in J \\ \,|x_i|> d_\infty (x) \textrm{ for } i \notin J, \end{matrix*} 
\end{cases}\label{eq:full_d_infty}
\end{equation}
where $\hat{x}$ is the projection of $x$ on the subspace $\{ x_j =0 : j\in J \}$, and $|J|$ is the number of elements of $J$. Note that $J$ can be empty too, but its complement must be nonempty. Also note that 
\[
|\hat{x}_i| =|x_i|> d^n_\infty (x) = d^{n-|J|}_\infty (\hat{x})
\]
implies that the last formula in \eqref{eq:d_infty} should be applied to compute $d^{n-|J|}_\infty (\hat{x})$. 
Furthermore, let us mention that the regions corresponding to each subset of indices $J\subset \{ 1, 2, \dots , n\}$ have a nonempty interior. In other words, the set of points in the exterior of $B_1$ whose $d_\infty$-closest point on $\partial B_1$ corresponds to singular points of $\partial K^\circ$ with some given normal cone dimension has nonempty interior.  

It is worth mentioning that for $x$ outside the unit sphere, $d_\infty (x,\partial B_1)$  can also be calculated by noting that $B_1$  touches $K_x$  at only one point, hence their intersection point is the closest point on $K_x$  to the origin with respect to the Euclidean distance. This idea can also be used to compute the distance to $\partial B_1$ with respect to other norms, provided that we can characterize the projection of the origin on $K_x$ for a given $x$. However, the more systematic approach presented above is better suited to generalization and in principle can  be used to compute the distance functions to the boundary of more general domains.  

Finally, we examine the set of differentiability of $d_\infty (\cdot, \partial B_1 )$. For $|x|<1$  the only points of nondifferentiability are on the hyperplanes $x_i =0$  for some $i$ . These are exactly the points with more than one $d_\infty$-closest point on $\partial B_1$. On the other hand, for $|x| >1$ the points with some $x_i =0$ are not among the points of nondifferentiability, since at these points the formula for $d_\infty$ does not contain $x_i$. It is also easy to see that in the formula $\eqref{eq:full_d_infty}$ for $d_\infty$, the term inside the square root has a positive lower bound (given by $(n-|J|)/2$) in its corresponding domain. So, the only possible points of nondifferentiability when $|x|>1$  are the points on the boundaries of the different regions over which $d_\infty$ is given by different formulas in $\eqref{eq:full_d_infty}$. However, a closer look reveals that the first derivative of $d_\infty$ exists at these points. Since for $j\in J$ we have 
\[
D_j d_\infty  =  \frac{1}{n-|J|}\Bigl[ \mathrm{sgn}(x_j) - \frac{\mathrm{sgn}(x_j)|\hat{x}|_{1}-(n-|J|)x_j}{\Big(|\hat{x}|_{1}^{2}+n-|J|-(n-|J|)|\hat{x}|^{2}\Big)^{1/2}}\Bigr],
\]
which becomes $0$ at $|x_j|=d_\infty (x)$---consistent with the value of $D_j d_\infty $ from the side with $|x_j| > d_\infty$---because at  $|x_j|=d_\infty (x)$ we have  
\[
\Big(|\hat{x}|_{1}^{2}+n-|J|-(n-|J|)|\hat{x}|^{2}\Big)^{1/2} = |\hat{x}|_1 - (n-|J|) d_\infty (x) = |\hat{x}|_1 - (n-|J|)|x_j|.
\]
But the second derivative of $d_\infty $ does not exist at these points, as can be easily seen by differentiating once more with respect to $x_j$. 
\end{example}

\begin{example}  \label{exterior-2-balls} Let us consider the exterior of two adjacent unit disks in $\mathbb{R}^2$, and examine the $d_\infty$ distance to their boundary. To be concrete, let 
\[
U = \{ x \in \mathbb{R}^2  : |x-(1,0)|>1 \textrm{ and } |x-(-1,0)|>1 \}.
\]
Then it is easy to see that when $x_1 =0$  and $x_2>2$ we have 
\[
d_\infty (x, \partial U) = |x_2 -1|.
\]
So although $x$ has two closest points on the boundary, namely $(\pm 1 , 1)$, $d_\infty$ is smooth around $x$. As can be seen in Figure \ref{fig:2-disks}, the two $d_\infty$-closest points $y,y'$ to $x$ lie on the same face of the square $K_x$. 

Note that $U$ can be turned into a domain with smooth boundary by connecting the two disks at their intersection point through a small canal. The same phenomenon still occurs.

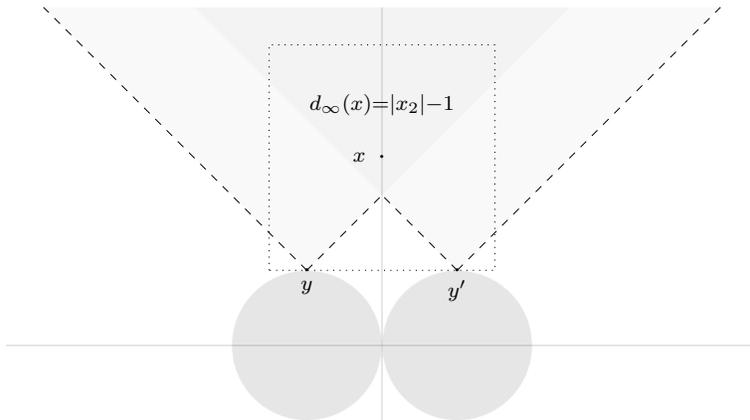
\begin{figure}
\begin{tikzpicture}

\draw [line width=0.2pt, opacity=0.01, fill=black, fill opacity=0.1] (1.,0.) circle (1.cm); 

\draw [line width=0.2pt, opacity=0.01, fill=black, fill opacity=0.1] (-1.,0.) circle (1.cm); 

\draw[color=black,dotted] (-1.5,1) rectangle (1.5,4);
\draw (0,2.5) node[anchor=center] {$\cdot$} node[anchor=east] {$\scriptstyle{x}\;$}; 
\draw (-1,1) node[anchor=center] {$\cdot$} node[anchor=north] {$\scriptstyle{y}$}; 
\draw (1,1) node[anchor=center] {$\cdot$} node[anchor=north] {$\scriptstyle{y'}$}; 


\draw[line width=0.1pt, opacity=0.2] (0,-1) -- (0,4.5);
\draw[line width=0.1pt, opacity=0.2] (-5,0) -- (5,0);

\draw [line width=0.2pt, dashed] (1,1) -- (4.5,4.5);
\draw [line width=0.2pt, dashed] (1,1) -- (0,2);
\fill [black,fill opacity=0.025] (1,1) -- (4.5,4.5) -- (-2.5,4.5) -- cycle;

\draw [line width=0.2pt, dashed] (-1,1) -- (0,2);
\draw [line width=0.2pt, dashed] (-1,1) -- (-4.5,4.5);
\fill [black,fill opacity=0.025] (-1,1) -- (2.5,4.5) -- (-4.5,4.5) -- cycle;

\draw (0,3.2) node[anchor=center] {$\scriptstyle{d_\infty (x) = |x_2|-1}$}; 

\end{tikzpicture}

\caption{\protect\label{fig:2-disks}The distance function $d_{\infty}(x, \partial U)$ is given by $|x_2|-1$ in the shaded region. In the
 region with darker shade the same formula holds, but here each point has two $d_{\infty}$-closest points on the boundary. 
The points $y,y'$ are the $d_{\infty}$-closest points to $x$. Note that
in this case the intersections of $K_{x}$ with $\partial U$ are
at non-vertex points of $\partial K_{x}$, but lie on the same face of it.}
\end{figure}
\end{example}

\begin{rem*}
    Let us summarize our observations in these examples: 
    \begin{itemize}
        \item[\textrm{(i)}] The distance function may or may not be differentiable at points with more than one closest point on the boundary, in contrast to the case of smooth strictly convex sets. 
        \item [\textrm{(ii)}] The signed distance function may not be even differentiable on the boundary itself, in contrast to the case of smooth strictly convex sets. 
        \item [\textrm{(iii)}] The distance function may fail to be smooth when the face of the polytope corresponding to the closest point changes. This region can be a large subset of the singular set of the distance function. 
    \end{itemize}
\end{rem*}

\bibliographystyle{plainnat-reversed}
\bibliography{Bibliography-Nov-2024.bib}

\end{document}